\documentclass[12pt,reqno]{amsart}

\usepackage{amsmath,amsfonts,amsthm,amssymb,amsxtra, mathrsfs
}
\usepackage{bbm} 
\usepackage{hyperref} 

\newtheorem{theorem}{Theorem}

\newtheorem{lemma}[theorem]{Lemma}
\newtheorem{corollary}[theorem]{Corollary}

\theoremstyle{definition}

\theoremstyle{remark}

\newtheorem{remark}[theorem]{Remark}

\newcommand\eps\varepsilon

\renewcommand\Re{\mathop{\mathrm{Re}}\nolimits}

\begin{document}

\title[Convex functionals on Bergman spaces]{Sharp stability of convex functionals on weighted Bergman spaces with applications}

\author{Petar Melentijevi\'{c}}
\address[Studentski trg 16, 11000 Beograd, Serbia]{Studentski trg 16, 11000 Beograd, Serbia}
\email{petar.melentijevic@matf.bg.ac.rs}

\begin{abstract}
	Recently, Kulikov (\cite{Ku}) has shown that certain convex functionals on weighted Bergman spaces are maximized by reproducing kernels. We show a sharp quantitative stability of these estimates with the optimal norm and the exponent and an explicit constant asymptotically sharp in both directions ($\alpha\rightarrow -1$ and $\alpha\rightarrow +\infty$). Several applications of this result include recovering the appropriate result for Fock spaces, interpretation to Cauchy wavelets, and the Hardy space counterpart for functionals induced by increasing function. In addition, we prove a higher-dimensional analog of the main result assuming that all convex functionals on the weighted Bergman space $\mathcal{A}^2_{\alpha}(\mathbb{B}_n)$ attain their extrema in reproducing kernels.
\end{abstract}

\maketitle

\section{Introduction and main result}
We will be working with several spaces of holomorphic functions in the unit disk $\mathbb{D}=\{z\in\mathbb{C}: |z|<1\}.$ The weighted Bergman space $\mathcal{A}^{p}_{\alpha}(\mathbb{D}),$ with $\alpha>-1$ is defined as the space of all functions $f,$ holomorphic in $\mathbb{D}$ such that
$$\|f\|_{\mathcal{A}^{p}_{\alpha}}:=\frac{\alpha+1}{\pi}\int_{\mathbb{D}}|f(z)|^p(1-|z|^2)^{\alpha}dA(z)<+\infty,$$
where $dA(z):=dxdy$ is the two-dimensional Lebesgue measure. Detailed exposition on the topic of Bergman spaces is given in \cite{HeKoZhu}. The Hardy space $H^p$ consists of holomorphic functions $f$ such that 
$$\|f\|_{H^p}:=\lim_{r\rightarrow 1^{-}}\int_{0}^{2\pi}|f(re^{\imath t})|^p \frac{dt}{2\pi}$$
is finite.
It is well-known that 
$\lim_{\alpha\rightarrow -1^{+}}\|f\|_{\mathcal{A}^{p}_{\alpha}}=\|f\|_{H^p},$
which provides the possibility of translating some results from the weighted Bergman to the Hardy space setting (\cite{Zhu}). 
By
$$\mu(\Omega):=\int_{\Omega}\frac{dA(z)}{(1-|z|^2)^2}$$
we denote M\"obius invariant hyperbolic measure of the set $\Omega\subset\mathbb{D}.$ In our analysis, we will estimate the hyperbolic measure of the super-level sets $A_{t}:=\{z\in \mathbb{D}: u(z)>t\},$ where $u(z):=|f(z)|^p(1-|z|^2)^{\alpha+2}.$

Contractive inclusions between different weighted Bergman spaces have raised a great interest in recent investigations in function theory. Motivated by the Wehrl's entropy conjecture for the affine-linear group (\cite{LiSo3}) in mathematical physics,  Lieb and Solovej posed a question whether the norms $\|f\|_{\mathcal{A}^{p\alpha}_{\alpha}}$ are decreasing in $p.$ Earlier, Pavlovi\'{c} in \cite{Pa} had asked if Carleman inequality can be interpolated for non-integer values, while Brevig-Ortega-Cerda-Seip-Zhao in paper \cite{BrOrCeSeZh} give some evidence about Pavlovi\'{c}'s conjecture and broader perspective of its applications. It was recently solved by Kulikov in \cite{Ku} via an ingenious adaptation of a technique invented in \cite{NiTi} to the hyperbolic setting in the unit disk. In fact, he proved a more general result, which reads as follows:
Let $G:[0,+\infty)\rightarrow \mathbb{R}$ be a convex function. Then the maximum value of 
$$\int_{\mathbb{D}} G(|f(z)|^p(1-|z|^2)^{\alpha+2})d\mu(z)$$
is achieved for $f(z)=1,$ subject to the condition that $f \in \mathcal{A}^p_{\alpha}(\mathbb{D})$ and $\|f\|_{\mathcal{A}^p_{\alpha}}=1.$  For the sake of brevity we call these functionals $\mathbf{convex}.$  In the case of the functions of the Hardy space $H^p$ under the constraint $\|f\|_{H^p}=1$ the same conclusion is proved for every increasing function $G.$ Thus, applying this result to $G(t)=t^r,$ the contractive embeddings
$$\|f\|_{\mathcal{A}^{q}_{\beta}}\leqslant \|f\|_{\mathcal{A}^{p}_{\alpha}}\leqslant \|f\|_{H^r},$$
is easily implied for  $\frac{\beta+2}{q}=\frac{\alpha+2}{p}=r$ and $0<p<q.$
An interesting higher-dimensional extension of this result is given in \cite{Kal}. In \cite{Lin} and \cite{Mel} sharp contractive embeddings between Hardy and Dirichlet or Besov spaces and hypercontractive estimates for weighted Bergman spaces are proved via Kulikov's inequalities. Let us say that Wehrl entropy conjectures have been investigated in a sequence of papers (\cite{Li}, \cite{Lu}, \cite{LiSo1}, \cite{LiSo2}, \cite{Sc}, \cite{We}, \cite{Wo}) using different methods from those in \cite{NiTi} (and \cite{Fr}, \cite{RaTi}), which is, as far as the author knows, the first paper which uses the isoperimetric inequality in this topic.   

Ramos and Till in \cite{RaTi} gave an important application of the method of \cite{NiTi}  , where the sharp concentration inequality for $f \in \mathcal{A}^{2}_{\alpha}$ is proved. Namely, for a set $\Omega$ of prescribed hyperbolic measures $\mu(\Omega)=s,$ the quantity 
$$R(f,\Omega)=\frac{\frac{\alpha+1}{\pi}\int_{\Omega}|f(z)|^2(1-|z|^2)^{\alpha}dA(z)}{\|f\|^2_{\mathcal{A}^2_{\alpha}}}$$
satisfies the following sharp estimate
$$R(f,\Omega)\leqslant R(1,B_{r})=1-\big(1+\frac{s}{\pi}\big)^{-\alpha-1}.$$
Here, $B_{r}$ is a disk centered at the origin such that $\mu(B_r)=s.$ The equality is achieved if and only if $f$ is a multiple of some reproducing kernel $f_w(z)=(1-\overline{w}z)^{-\alpha-2}$ and $\Omega$ is a disk centered in $w$ of hyperbolic measure $s.$
This is further generalized in  hyperbolic higher dimension setting in \cite{KalRam}.\\

Very recently, stability of the local (concentration) estimates has become an active topic of research. In \cite{GGRT} the authors proved a sharp quantitative form of concentration inequalities in Fock spaces that gives sharp stability of the Faber-Krahn inequality for a short-time Fourier transform. For the foundations of time-frequency analysis, see \cite{grochenig}. The ideas of \cite{GGRT} are further employed in more technically demanding estimates for the weighted Bergman spaces, asymptotically sharp in both directions ($\alpha\rightarrow -1^{+}$ and $\alpha\rightarrow +\infty$). This result not only recovers the known estimate for the short-time Fourier transform but also gives a novel concentration estimate for a function from Hardy spaces. This was done in \cite{GKMR}. An improved bound for the Hilbert-Schmidt norm of localization operators in the time-frequency plane is the subject of \cite{NiRi}.

In \cite{FrNiTi} Frank, Nicola and Tilli proved a sharp quantitative version of the generalized Wehrl conjecture for the short-time Fourier transform. In particular,  
they provided a stability estimate for the logarithmic Sobolev inequality in Fock spaces, first proved in \cite{Ca}. This inequality is an important example of improved hypercontractivity, which is also considered in \cite{Ja}, \cite{Ja2} (partial solution to the hypercontractivity conjecture in \cite{Ja} is given in \cite{Mel}). The literature on many other instances of logarithmic Sobolev inequalities and their quantitative forms can be found in \cite{DoEsFiFrLo}. The approach in \cite{FrNiTi} uses majorization theory, local estimates, and their stability versions. The analogous problem in the case of spaces of holomorphic polynomials is addressed in \cite{GFOC}. These results motivate a similar question of considering convex functionals on weighted Bergman spaces, which will be our main task. 
Our first result reads as follows:
\begin{theorem}
Let $G:[0,1]\rightarrow\mathbb{R}$ be a continuous and convex non-linear function such that $G(0)=0.$ Then, for every $f\in\mathcal{A}^2_{\alpha}(\mathbb{D})$ there exists a constant $c_{G}$ depending on $G$ and $\alpha>-1$ for which 
$$\int_{\mathbb{D}} G(|f(z)|^2(1-|z|^2)^{\alpha+2})d\mu(z)\leqslant                                               \int_{\mathbb{D}} G((1-|z|^2)^{\alpha+2})d\mu(z)-c_{G}\min_{|c|=\|f\|_{\mathcal{A}^2_{\alpha}}}\frac{\|f-cf_w\|^2_{\mathcal{A}^2_{\alpha}}}{\|f\|^2_{\mathcal{A}^2_{\alpha}}}. $$
\end{theorem}
In particular, taking $G(t)=t^{\frac{p}{2}}, p>2$ we obtain the following corollary:
\begin{corollary}
There exists a computable constant $c_{p,\alpha}$ for $p>2$ and $\alpha>-1$ such that
\begin{align*}
\|f\|^p_{\mathcal{A}^p_{\beta}}
&\leqslant \|f\|^p_{\mathcal{A}^2_{\alpha}}\bigg(1-c_{p,\alpha}\min_{|c|=\|f\|_{\mathcal{A}^2_{\alpha}}}\frac{\|f-cf_w\|_{\mathcal{A}^2_{\alpha}}^2}{\|f\|_{\mathcal{A}^2_{\alpha}}^2}\bigg)^{\frac{p}{2}}\\
&\leqslant \|f\|^p_{\mathcal{A}^2_{\alpha}}-c_{p,\alpha}^\frac{p}{2}\min_{|c|=\|f\|_{\mathcal{A}^2_{\alpha}}}\frac{\|f-cf_w\|_{\mathcal{A}^2_{\alpha}}^p}{\|f\|_{\mathcal{A}^2_{\alpha}}^p},
\end{align*}
where $\beta+2=\frac{p(\alpha+2)}{2}.$
\end{corollary}
We can prove Theorem 1 using a more direct approach which works in a slightly general situation. Namely, we have the following theorem :
\begin{theorem}
Let $A$ be a positive operator with the unit trace norm and $G$ be a continuous convex, non-linear function. Then, for some computable absolute constant $c_A>0,$ we have
$$\int_{\mathbb{D}} G(\langle \kappa_z, A\kappa_z\rangle)d\mu(z)\leqslant                                               \int_{\mathbb{D}} G((1-|z|^2)^{\alpha+2})d\mu(z)-c_{A}\inf_{\xi\in \mathbb{D}}\|A-\langle \cdot, \kappa_{\xi}\rangle \kappa_{\xi}\|_{\mathcal{S}_1},$$
where $\mathcal{S}_1$ denotes the trace norm and $\kappa_{\xi}$ normalized Bergman reproducing kernel in $\mathcal{A}^2_{\alpha}(\mathbb{D}
).$
\end{theorem}
We turn now our attention to higher dimensional weighted Bergman spaces. Let $\mathcal{A}^p_{\alpha}(\mathbb{B}_n)$ be the space of the holomorphic functions in the unit ball $\mathbb{B}_n$ of $\mathbb{C}^n$ such that
$$
\|f\|_{\mathcal{A}^p_{\alpha}(\mathbb{B}_n)}:=
\bigg(\frac{\Gamma(\alpha+n+1
)}{n!\Gamma(\alpha+1)}\int_{\mathbb{B}_n}|f(z)|^p(1-|z|^2)^{\alpha+n+1}d\mu_n(z)\bigg)^{\frac{1}{p}}<+\infty,$$
where $\alpha>-1,$ $d\mu_n(z)=\frac{dV(z)}{(1-|z|^2)^{n+1}}$ and $dV(z)$ is the usual Lebesgue volume measure. In \cite{NiRiTi} it is proved that if
$$\int_{\mathbb{B}_n} G(|f(z)|^p(1-|z|^2)^{\alpha+n+1})d\mu_n(z)\leqslant                                               \int_{\mathbb{B}_n} G((1-|z|^2)^{\alpha+n+1})d\mu_n(z)$$
is true for $\mathbf{all}$ convex functions $G$, then the extremizers are exactly multiples of \\
$(1-\langle z, w\rangle)^{-\frac{2(\alpha+n+1)}{p}}$ and the quantities $\int_{\mathbb{B}_n} G(|f(z)|^p(1-|z|^2)^{\alpha+n+1})d\mu_n(z)$ and\\
$\int_{\mathbb{B}_n} G((1-|z|^2)^{\alpha+n+1})d\mu_n(z)
$ are close only if $f$ is close to the set of extremizers. Here, we will get the optimal stability result:
\begin{theorem}
Let
$$\int_{\mathbb{B}_n} G(|f(z)|^2(1-|z|^2)^{\alpha+n+1})d\mu_n(z)\leqslant                                               \int_{\mathbb{B}_n} G((1-|z|^2)^{\alpha+n+1})d\mu_n(z)$$
holds for every convex function $G$ and $f\in \mathcal{A}^2_{\alpha}(\mathbb{B}_n).$ Then we have
\begin{align*}
&\int_{\mathbb{B}_n} G(|f(z)|^2(1-|z|^2)^{\alpha+n+1})d\mu_n(z)\\
\leqslant
&\int_{\mathbb{B}_n} G((1-|z|^2)^{\alpha+n+1})d\mu_n(z)-c_{G}\min_{|c|=\|f\|_{\mathcal{A}^2_{\alpha}(\mathbb{B}_n)}}\frac{\|f-cf_w\|^2_{\mathcal{A}^2_{\alpha}(\mathbb{B}_n)}}{\|f\|^2_{\mathcal{A}^2_{\alpha}(\mathbb{B}_n)}}  
\end{align*}
for some absolute constant $c_G$ depending only on the function $G$ and $f_w(z)=(1-\langle z, w\rangle)^{-\alpha-n-1}.$ 
\end{theorem}

In the case of Hardy spaces, we consider the functionals induced by increasing functions, and the result takes the following form: 
\begin{theorem}
Let $G:[0,1]\rightarrow\mathbb{R}$ be a continuous and strictly increasing or convex and increasing function such that $G(0)=0.$ Then, for every $f\in H^2,$ there exists a constant $c_{G}$ depending on $G$ for which 
$$\int_{\mathbb{D}} G(|f(z)|^2(1-|z|^2))d\mu(z)\leqslant                                               \int_{\mathbb{D}} G((1-|z|^2))d\mu(z)-c_{G}(1-T), $$
where $T\leqslant 1$ is the constant such that
$$|f(z)|^2(1-|z|^2)\leqslant T\|f\|^2_{H^2}.$$
\end{theorem}
The paper is organized as follows. The second section describes the fundamental relations between the main notions in the context of the functionals on the weighted Bergman spaces that are induced by convex functions. The proof of stability through majorization theory and stability estimates of local inequalities with applications to Fock and Hardy spaces (using the technique from \cite{FrNiTi}) are given in the third section. Two proofs of the general result for the positive operator with finite trace norm using the methods from \cite{FrNiTi} and \cite{NiRiTi} are the main content of the fourth section. In the same section, we give the wavelet transform interpretation. The fifth section is devoted to the proof of the stability of the convex functionals on $\mathcal{A}^p_{\alpha}(\mathbb{B}_n)$ assuming that the analog of Kulikov's inequality in a higher-dimensional setting holds for every convex function. In the last, sixth section we prove the result for Hardy space setting and functionals given by strictly increasing functions. 

\section{Convex functionals on weighted Bergman spaces}

Consider $f\in \mathcal{A}^p_{\alpha}(\mathbb{D})$ with the norm $1$. Denote
\begin{equation*}
u(z)=|f(z)|^p(1-|z|^2)^{\alpha+2}.
\end{equation*}
Define $\rho(t):=\mu(\{z\in\mathbb{D}: u(z)>t\})$ and $\rho_0(t):=\mu(\{z\in\mathbb{D}: v_{\alpha}(z)>t\}).$
For an arbitrary measurable $\Omega \subset \mathbb{D},$ 
 the following Ramos-Tilli localization inequality holds for all $p>0$ (as proved in \cite{RaTi} or \cite{KuNiOCTi})
$$\int_{\Omega} u(z)d\mu(z)\leqslant \int_{\Omega^*} v_\alpha(z)d\mu(z),$$
where $\Omega^*$ is the disk centered in $0$ such that $\mu(\Omega^*)=\mu(\Omega).$

  In \cite{Fr} the author observed that the concentration estimate and the inequality for convex functionals are equivalent. This is a consequence of an integral version of Karamata's majorization inequality (which can be found in \cite{HaLiPo},\cite{Si3} or, in a more convenient form for our setting in \cite{AlTrLi}). This will be, along with the quantitative version of the concentration estimate, the crucial idea in proving our theorem. Therefore, we will first recover Kulikov's result from the Ramos-Tilli local estimate following the approach from \cite{FrNiTi}. Taking an arbitrary $t>0$ and the super-level set $A_t:=\{u(z)>t\}$ for $\Omega$ in Ramos-Tilli inequality, we have:
\begin{align*}
&\int_{\mathbb{D}}(u(z)-t)_+\:\! d\mu(z)=\int_{A_t} (u(z)-t) d\mu(z)\\
=&\int_{A_t} u(z)d\mu(z)-t\mu(A_t)\leqslant \int_{A_t^*} v_\alpha(z)d\mu(z)-t \mu(A_t^*),
\end{align*}
where $A_t^{*}$ stands for the disk with center in $0$ and hyperbolic measure equal to $\mu(A_t).$
Further,
\begin{multline}
\label{mainineq}
 \int_{A_t^*} v_\alpha(z)d\mu(z)-t\mu(A_t^*) = \int_{A_t^*} (v_\alpha(z)-t)d\mu(z)\\
 \leqslant \int_{A_t^*} (v_\alpha(z)-t)_+\:\! d\mu(z)\leqslant\int_{\mathbb{D}}(v_\alpha(z)-t)_+\:\! d\mu(z).
\end{multline}
By using
\begin{equation*}
G(u)=G'(0)u+\int_0^{+\infty} (u-t)_+ G''(t)dt
\end{equation*}
for a convex $G$ such that $G(0)=0$ and $G'(0)>-\infty,$ we get
\begin{multline*}
\int_{\mathbb{D}}G(u(z))d\mu(z)=\int_{\mathbb{D}} u(z)d\mu(z)\cdot G'(0)
+\int_0^{+\infty}\! \biggl(\int_{\mathbb{D}}(u(z)-t)_+\:\! d\mu(z)\!\biggr)G''(t)dt\\
\leqslant \int_{\mathbb{D}}v_\alpha(z)d\mu(z)\cdot G'(0)
+\int_0^{+\infty}\!\biggl(\int_{\mathbb{D}}(v_\alpha(z)-t)_+\:\! d\mu(z)\!\biggr)G''(t)dt
=\int_{\mathbb{D}}G(v_\alpha(z))d\mu(z).
\end{multline*}
Recall that a convex function has the left and right derivatives at every point, and the second derivative $G''(t)dt$ is a measure that does not have to be absolutely continuous. 
If $G'(0)=-\infty,$ then, by the above, we have 
\begin{equation*}
  \int_{\mathbb{D}}G_\varepsilon(u(z))d\mu(z)
  \leqslant \int_{\mathbb{D}}G_\varepsilon(v_\alpha(z))d\mu(z)
\end{equation*}
for $G_\varepsilon(u)=\max\{G(u),-\frac{u}{\varepsilon}\},$ which has a bounded derivative almost everywhere. Since this family converges monotonically to $G$ in the segment $[0,T],$ as $\varepsilon\to 0+,$  we conclude the proof in this case.  \\

In order to achieve a refinement of the inequality \eqref{mainineq}, we will determine the difference of the integrals of $v_{\alpha}(z)-t$ and its positive part for fixed $t>0$ in the appropriate sets. Namely, we start from the equation:
\begin{multline}
\label{jednakost}
  \int_{A_t^*}(v_\alpha(z)-t)d\mu(z)\\
  =\int_{\mathbb{D}}(v_\alpha(z)-t)_+\:\! d\mu(z)
  - \int_{B_t \setminus A_t^*} (v_\alpha(z)-t)_+\:\! d\mu(z)
  - \int_{A_t^*\setminus B_t} (v_\alpha(z)-t)_-\:\! d\mu(z),
\end{multline}
where $B_t$ stands for the disk $\{z\in\mathbb{D}: v_{\alpha}(z)>t\}.$ Both $A_t^*$ and $B_t$ are disks centered in $0;$ therefore, if
$\mu(A_t^*)>\mu(B_t)$ then
$B_t\setminus A_t^*=\emptyset,$ otherwise ($\mu(A_t^*)\leqslant\mu(B_t)$) $A_t^*\setminus B_t =\emptyset.$ In the first case, we get:
\begin{equation*}
  \int_{A_t^*}(v_\alpha(z)-t)d\mu(z)=
  \int_{\mathbb{D}}(v_\alpha(z)-t)_+ \:\! d\mu(z)
  - \int_{A_t^*\setminus B_t} (v_\alpha(z)-t)_- \:\! d\mu(z).
\end{equation*}
Now, we will calculate the second integral on the right hand side. Introducing polar coordinates, we have
\begin{equation*}
\int_{A_t^*\setminus B_t} (v_\alpha(z)-t)_- \:\!d\mu(z)
=2\pi\int_{r(B_t)}^{r_t} \Bigl(t-(1-\rho^2)^{\alpha+2}\Bigr)\frac{\rho d\rho}{(1-\rho^2)^2}.
\end{equation*}
While
\begin{equation*}
  r^2(B_t)=1-t^{\!\frac1{\alpha+2}},
\end{equation*}
with $r_t$ we denote the radius of $A_t^{*}$, i.e. we have:
\begin{equation*}
  \mu(A_t^*)=\mu(A_t)=\pi\Bigl(\frac1{1-r_t^2}-1\Bigr).
\end{equation*}
We find:
\begin{align*}
&\pi\int_{r^2(B_t)}^{r^2_t}\Bigl(t-(1-y)^{\alpha+2}\Bigr)\frac{dy}{(1-y)^2}\\
&=\pi\Bigl(\frac{t}{1-y}+\frac{(1-y)^{\alpha+1}}{\alpha+1}\Bigr)\Big|_{y=r^2(B_t)}^{r^2_t}\\
&=\pi t\Bigl(\frac1{1-r^2_t}-\frac1{1-r^2(B_t)}\Bigr)+\frac{\pi}{\alpha+1}\bigg((1-r^2_t)^{\alpha+1}
-(1-r^2(B_t))^{\alpha+1}\bigg)\\
&=\pi t\Bigl(\frac{\mu(A_t)}{\pi}+1-t^{-\frac1{\alpha+2}}\Bigr)
+\frac{\pi}{\alpha+1}\bigg(\Bigl(1+\frac{\mu(A_t)}{\pi}\Bigr)^{\!\!-\alpha-1}-t^{\frac{\alpha+1}{\alpha+2}}\bigg)\\
&=t(\mu(A_t)+\pi)+\frac{\pi}{\alpha+1}\bigg[\Bigl(1+\frac{\mu(A_t)}{\pi}\Bigr)^{\!\!-\alpha-1}
-(\alpha+2)t^{\frac{\alpha+1}{\alpha+2}}\bigg],
\end{align*}
from which
$$
\int_{A_t^*\setminus B_t\!\!} (v_\alpha(z)-t)_- \:\!d\mu(z)
=t(\mu(A_t)+\pi)+\frac{\pi}{\alpha+1}\bigg[\Bigl(1+\frac{\mu(A_t)}{\pi}\Bigr)^{\!\!-\alpha-1}
-(\alpha+2)t^{\frac{\alpha+1}{\alpha+2}}\bigg].
$$

Hence, in the case $\mu(A_t)>\mu(B_t),$ we get:
\begin{align*}
  &\int_{A_t^*}(v_\alpha(z)-t)d\mu(z)\\
  =&\int_{\mathbb{D}}(v_\alpha(z)-t)_{+}d\mu(z)-\bigg[t(\pi+s)+\frac{\pi}{\alpha+1}\bigg(\big(1+\frac{s}{\pi}\big)^{-\alpha-1}-(\alpha+2)t^{\frac{\alpha+1}{\alpha+2}}\bigg)\bigg],
\end{align*}
where $s=\mu(A_t).$  We proceed quite analogously for $\mu(A_t)\leqslant\mu(B_t)$ and the only difference is the sign of the last summand in the previous identity. More precisely, the result of this calculation is given by the following lemma:
\begin{lemma}
\begin{multline}
\label{JEDN}
  \int_{A_t^*}(v_\alpha(z)-t)d\mu(z)\\
  =\int_{\mathbb{D}}(v_\alpha(z)-t)_{+}d\mu(z)-\bigg|t(\pi+s)+\frac{\pi}{\alpha+1}\bigg(\big(1+\frac{s}{\pi}\big)^{-\alpha-1}-(\alpha+2)t^{\frac{\alpha+1}{\alpha+2}}\bigg)\bigg|.
\end{multline}
\end{lemma} 
Note that from differential inequality $\frac{-\rho'(t)}{\pi+\rho(t)}\geqslant \frac{1}{(\alpha+2)t},$ integrating from $t$ to $T$ we get $(1+\frac{\rho(t)}{\pi})^{-\alpha-2}\leqslant \frac{t}{T}.$ Therefore, $A(s,t)$ is positive for $\frac{t}{T}>(1+\frac{s}{\pi})^{-\alpha-2}>t$ and negative for $(1+\frac{s}{\pi})^{-\alpha-2}<t.$

\section{Optimal stability for $\mathcal{A}^2_{\alpha}$- proof via concentration estimates and majorization theory} 
In this section, we concentrate on the case $p=2.$ The pointwise inequality
$$u(z)\leqslant \|f\|_{\mathcal{A}^2_{\alpha}}=1$$
is known to be true for every $f\in\mathcal{A}^2_{\alpha}.$
The equality is attained if and only if $f$ is equal to the multiple of the reproducing kernel $f_w(z)=(1-\overline{w}z)^{-\alpha-2}.$ Let us suppose that $\max_{\mathbb{D}}u(z)=T<1.$
In \cite{GKMR} it is proved that, for $f\in \mathcal{A}_\alpha^2$ there holds the estimate
\begin{equation}\label{prva}
  1-T\leqslant 2(1-\sqrt{T})=\min_{|c|=\|f\|_{\mathcal{A}^2_{\alpha}}}\frac{\|f-cf_w\|^2_{\mathcal{A}^2_{\alpha}}}{\|f\|^2_{\mathcal{A}^2_{\alpha}}}\leqslant M_\alpha(s)\delta(f;\Omega,\alpha)
\end{equation}
where 
$$M_\alpha(s)=C\bigg(1+\frac{\alpha+2}{\alpha+1}\bigg[\big(1+\frac{s}{\pi}\big)^{\alpha+1}-1\bigg]\bigg),$$
$$\delta(f;\Omega,\alpha)=1-\frac{\int_\Omega|f(z)|^2(1-|z|^2)^{\alpha+2}d\mu(z)}{\int_{\Omega^*} v_{\alpha}d\mu(z)}$$
and
$$v_\alpha(z)=(1-|z|^2)^{\alpha+2}$$
for $s=\mu(\Omega),\Omega\subset \mathbb{D}.$
From \eqref{prva}, we have
\begin{equation*}
  \frac{1-T}{M_\alpha(s)}\leqslant 1-\frac{\int_\Omega |f(z)|^2(1-|z|^2)^{\alpha+2}d\mu(z)}{\int_{\Omega^*}v_\alpha(z)d\mu(z)}
\end{equation*}
i.e.\ 
\begin{equation}\label{druga}
\int_\Omega |f(z)|^2(1-|z|^2)^{\alpha+2}d\mu(z)
\leqslant  \biggl(1-\frac{1-T}{M_\alpha(s)}\biggr)\int_{\Omega^*}v_\alpha(z)d\mu(z).
\end{equation}

Inserting $\Omega:= A_t=\{u(z)>t\},$ for $u(z)=|f(z)|^2(1-|z|^2)^{\alpha+2},$ we get:
\begin{align*}
\int_{\mathbb{D}}(u(z)-t)_+\:\! d\mu(z)&=\int_{A_t}(u(z)-t)d\mu(z)
=\int_{A_t} u(z)d\mu(z)-t \mu(A_t)\\
&\leqslant \int_{A_t^*} v_\alpha(z)d\mu(z)\cdot \biggl(1-\frac{1-T}{M_\alpha(s)}\biggr)
-t \mu(A_t^*)\\
&=\biggl(1-\frac{1-T}{M_\alpha(s)}\biggr) \int_{A_t^*} (v_\alpha(z)-t)d\mu(z)
-t\frac{1-T}{M_\alpha(s)}\\
&\leqslant \biggl(1-\frac{1-T}{M_\alpha(s)}\biggr) \int_{A_t^*} (v_\alpha(z)-t)d\mu(z)\\
&\leqslant \biggl(1-\frac{1-T}{M_\alpha(s)}\biggr) \int_{A_t^*} (v_\alpha(z)-t)_{+} d\mu(z)\\
&\leqslant \biggl(1-\frac{1-T}{M_\alpha(s)}\biggr) \int_{\mathbb{D}} (v_\alpha(z)-t)_+\:\! d\mu(z).
\end{align*}

Now we divide the proof into two cases, depending on a certain relation between $s$ and $t$. 
\subsection{Case $(1+\frac{s}{\pi})^{\alpha+2}\leqslant \frac{2}{t}$}
We easily see that
\begin{align*}
&M_{\alpha}(s)=C\bigg(1+\frac{\alpha+2}{\alpha+1}\bigg[\bigg(1+\frac{s}{\pi}\bigg)^{\alpha+1}-1\bigg]\bigg)\\
&\leqslant C\bigg(1+\frac{\alpha+2}{\alpha+1}\big(2^{\frac{\alpha+1}{\alpha+2}}t^{-\frac{\alpha+1}{\alpha+2}}-1\big)\bigg)
\leqslant 2M_{\alpha}(\mu(B_t)),
\end{align*}
(we used the inequality $2^{\frac{\alpha+1}{\alpha+2}}\leqslant 1+\frac{\alpha+1}{\alpha+2}),$ from which we infer:
\begin{equation}
\label{stabil1}
\int_{\mathbb{D}} (u(z)-t)_+\:\! d\mu(z)
\leqslant \biggl(1-\frac{1-T}{2M_\alpha(\mu(B_t))}\biggr) \int_{\mathbb{D}} (v_\alpha(z)-t)_+\:\! d\mu(z).
\end{equation}
Since $v_\alpha(z)=(1-|z|^2)^{\alpha+2}>t$ if and only if
$|z|^2<1-t^{\frac{1}{\alpha+2}}=r^2(B_t),$ 
we get
$$\rho_0(t)=\mu(B_t)=\pi\Bigl(\frac1{1-\big(1-t^{-\frac1{\alpha+2}}\big)}-1\Bigr)=\pi\big(t^{-\frac{1}{\alpha+2}}-1\big),$$
and
$${M_\alpha(\mu(B_t))=\frac{C}{\alpha+1}\Bigl((\alpha+2)
t^{-\frac{\alpha+1}{\alpha+2}}-1\Bigr).}$$

\subsection{Case $(1+\frac{s}{\pi})^{\alpha+2}\geqslant \frac{2}{t}$}
Let us start from
\begin{multline*}
  \int_{A_t^*}(v_\alpha(z)-t)d\mu(z)=\\
  \int_{\mathbb{D}}(v_\alpha(z)-t)_+\:\! d\mu(z)
  - \int_{B_t\setminus A_t^*} (v_\alpha(z)-t)_+\:\! d\mu(z)
  - \int_{A_t^*\setminus B_t} (v_\alpha(z)-t)_-\:\! d\mu(z).
\end{multline*}
Since $M_\alpha(\mu(A_t^*))>M_\alpha(\mu(B_t)),$ we have
 $\mu(A_t^*)>\mu(B_t)$ and therefore
$B_t\setminus A_t^*=\emptyset,$ thus giving, by Lemma 3:
\begin{multline}
  \int_{A_t^*}(v_\alpha(z)-t)d\mu(z)\\
  =\int_{\mathbb{D}}(v_\alpha(z)-t)_{+}d\mu(z)-\bigg[t(\pi+s)+\frac{\pi}{\alpha+1}\bigg(\big(1+\frac{s}{\pi}\big)^{-\alpha-1}-(\alpha+2)t^{\frac{\alpha+1}{\alpha+2}}\bigg)\bigg].
\end{multline}
We also find 
\begin{align*}
K_{\alpha}(t):=\int_{\mathbb{D}}(v_\alpha(z)-t)_+ \:\!d\mu(z)
&=2\pi\int_0^{r^2(B_t)}\!\Bigl((1-\rho^2)^{\alpha+2}-t\Bigr)\frac{\rho d\rho}{(1-\rho^2)^2}\\
&=\pi \int_0^{r^2(B_t)}\!\Bigl((1-y)^{\alpha+2}-t\Bigr)\frac{dy}{(1-y)^2}\\
&=\pi\Bigl(-\frac{t}{1-y}-\frac{(1-y)^{\alpha+1}}{\alpha+1}\Bigr)\Big|_{y=0}^{r^2(B_t)}\\
&=\pi t\Bigl(1-\frac1{1-r^2(B_t)}\Bigr)+\frac{\pi}{\alpha+1}\bigg(1-(1-r^2(B_t))^{\alpha+1}\bigg)\\
&=\pi t\Bigl(1-t^{-\frac{1}{\alpha+2}}\Bigr)+\frac{\pi}{\alpha+1}\bigg(1-t^{\frac{\alpha+1}{\alpha+2}}\bigg).
\end{align*}
Note that the last expression is not larger than $\frac{\pi}{\alpha+1},$ since 
$$\pi t-\frac{\pi(\alpha+2)}{\alpha+1}t^{\frac{\alpha+1}{\alpha+2}}<\pi t-\pi t^{\frac{\alpha+1}{\alpha+2}}<0.$$

Therefore, \eqref{mainineq} implies
$$\int_{\mathbb{D}}(u(z)-t)_+ \:\!d\mu(z)
\leqslant \int_{\mathbb{D}}(v_\alpha(z)-t)_+ \:\!d\mu(z)-\bigg[t(\pi+s)+\frac{\pi}{\alpha+1}\bigg(\bigg(1+\frac{s}{\pi}\bigg)^{-\alpha-1}-(\alpha+2)t^{\frac{\alpha+1}{\alpha+2}}\bigg]\bigg)$$
$$=\bigg(1-\bigg[t(\pi+s)+\frac{\pi}{\alpha+1}\bigg(\bigg(1+\frac{s}{\pi}\bigg)^{-\alpha-1}-(\alpha+2)t^{\frac{\alpha+1}{\alpha+2}}\bigg)\bigg]\big(K_{\alpha}(t)\big)^{-1}\bigg)\int_{\mathbb{D}}(v_\alpha(z)-t)_+ \:\!d\mu(z).$$

Let us consider:
$$\Phi(s)=t(\pi+s)+\frac{\pi}{\alpha+1}\bigg(\bigg(1+\frac{s}{\pi}\bigg)^{-\alpha-1}-(\alpha+2)t^{\frac{\alpha+1}{\alpha+2}}\bigg).$$
This function is convex and attains its minimum in $s_0$ such that $\big(1+\frac{s_0}{\pi}\big)^{-\alpha-2}=t.$ Hence, it is increasing in $(\pi\big((\frac{k_{\alpha}}{t}\big)^{\frac{1}{\alpha+2}}-1),+\infty)$ and:
$$\Phi(s)\geqslant \Phi\big(\pi\big((\frac{k_{\alpha}}{t}\big)^{\frac{1}{\alpha+2}}-1\big)\big)=\frac{\pi t^{\frac{\alpha+1}{\alpha+2}}}{\alpha+1}\bigg((\alpha+\frac{3}{2})2^{\frac{1}{\alpha+2}}-\alpha-2\bigg)=:C_{\alpha} t^{\frac{\alpha+1}{\alpha+2}}.$$
It is not hard to see that $C_{\alpha}>0$ for all $\alpha>-1.$ Also, it is bounded from below for $\alpha$ close to $-1$ and $C_{\alpha}\sim \frac{\log 2-\frac{1}{2}}{\alpha+1}$ for $\alpha\rightarrow +\infty.$
This gives 
\begin{align}
&\notag \int_{\mathbb{D}}(u(z)-t)_+ \:\!d\mu(z)\\
&\notag \leqslant \bigg(1-C_{\alpha}t^{\frac{\alpha+1}{\alpha+2}}\big(K_{\alpha}(t)\big)^{-1}\bigg)\int_{\mathbb{D}}(v_\alpha(z)-t)_+ \:\!d\mu(z)\\
&\label{stabil2}\leqslant \bigg(1-C_{\alpha}(1-T)t^{\frac{\alpha+1}{\alpha+2}}\big(K_{\alpha}(t)\big)^{-1}\bigg)\int_{\mathbb{D}}(v_\alpha(z)-t)_+ \:\!d\mu(z).
\end{align}
Since
$$M_{\alpha}(\mu(B_t))=\frac{C\big(\alpha+2-t^{\frac{\alpha+1}{\alpha+2}}\big)}{(\alpha+1)t^{\frac{\alpha+1}{\alpha+2}}}\geqslant C t^{-\frac{\alpha+1}{\alpha+2}},$$
by integrating the inequality of the form \eqref{stabil1}, \eqref{stabil2} over $(0,T),$ for every convex function $G,$ we get
\begin{equation}
\label{stabil}
\int_{\mathbb{D}}G(u(z)) \:\!d\mu(z)\leqslant \int_{\mathbb{D}}G(v_{\alpha}(z)) \:\!d\mu(z)-\frac{C_{\alpha}(1-T)}{C}\int_{0}^{T}\frac{G''(t)}{K_{\alpha}(t)M_{\alpha}(\mu(B_t))}dt.
\end{equation}
Let us note that, by \eqref{prva}, every $1-T$ in the above equations can be exchanged with $\min_{|c|=\|f\|}\frac{\|f-cf_w\|^2}{\|f\|^2}=2(1-\sqrt{T}),$ but we use the first expression because of its simplicity.
This proves stability estimates for functionals on weighted Bergman spaces.\\
\subsection{Asymptotics $\alpha\rightarrow +\infty$ and stability of convex functionals acting on Fock space} Here, we will recover the stability estimates for the Fock spaces, proved in \cite{FrNiTi}, by limitting argument. We start from 
$$f(z):=\sum_{k=0}^{+\infty} a_k\sqrt{\frac{\pi^k}{k!}}z^k \in \mathcal{F}^2(\mathbb{C}),$$
where $\mathcal{F}^2(\mathbb{C})$ is the Fock space, the space of all entire functions such that
$$\|f\|_{\mathcal{F}^2(\mathbb{C})}:=\bigg(\int_{\mathbb{C}}|f(z)|^2e^{-\pi|z|^2}dA(z)\bigg)^{\frac{1}{2}}<+\infty.$$
It is easy to see that the condition $\|f\|_{\mathcal{F}^2(\mathbb{C})}=1$ is equivalent to $\sum_{k=0}^{+\infty}|a_k|^2=1.$ Now we use \eqref{stabil} with $u(z)=|f_{\alpha}(z)|^2(1-|z|^2)^{\alpha+2},$ with $f_{\alpha}(z)=\sum_{k=0}^{+\infty}a_k\sqrt{\frac{(\alpha+2)(\alpha+3)\cdot\dots\cdot(\alpha+k+1)}{k!}}z^k.$ This function belongs to $\mathcal{A}^2_{\alpha}(\mathbb{D})$ and $\|f_{\alpha}\|_{\mathcal{A}^2_{\alpha}(\mathbb{D})}=1.$
We have
\begin{align*}
\frac{\alpha+1}{\pi}\int_{\mathbb{D}}G(u(z))d\mu(z)&=\frac{\alpha+1}{\pi}\int_{\mathbb{D}}G(|f_{\alpha}(z)|^2(1-|z|^2)^{\alpha+2})\frac{dA(z)}{(1-|z|^2)^2}\\
&=\frac{\alpha+1}{\alpha}\int_{\sqrt{\frac{\alpha}{\pi}}\mathbb{D}}G\bigg(|f_{\alpha}\big(\frac{\sqrt{\pi} z}{\sqrt{\alpha}}\big)|^2\big(1-\frac{|z|^2}{\alpha}\big)^{\alpha+2}\bigg)\frac{dA(z)}{(1-\frac{|z|^2}{\alpha})^2}.
\end{align*}
Since
\begin{align*}
\big|f_{\alpha}\big(\frac{\sqrt{\pi}z}{\sqrt{\alpha}}\big)\big|^2&=\bigg|\sum_{k=0}^{+\infty}a_kz^k\sqrt{\frac{\pi^k(\alpha+2)(\alpha+3)\cdot\dots\cdot(\alpha+k+1)}{k!\alpha^k}}\bigg|^2\\
&\rightarrow \bigg|\sum_{k=0}^{+\infty}a_kz^k\sqrt{\frac{\pi^k}{k!}}\bigg|^2=|f(z)|^2,
\end{align*}
and
$$(1-\frac{\pi |z|^2}{\alpha})^{\alpha+2}\rightarrow e^{-\pi|z|^2}$$
as $\alpha$ approaches $+\infty,$ while $\sqrt{\frac{\alpha}{\pi}}\mathbb{D}$ tends to the whole $\mathbb{C},$
we have 
$$\frac{\alpha+1}{\pi}\int_{\mathbb{D}}G(u(z))d\mu(z)\rightarrow \int_{\mathbb{C}}G(|f(z)|^2e^{-\pi|z|^2})dA(z), \quad \alpha\rightarrow+\infty.$$
The extremal  $v_{\alpha}(z)=(1-|z|^2)^{\alpha+2}$ corresponds to $f(z)=1$ and we easily get:
$$\frac{\alpha+1}{\pi}\int_{\mathbb{D}}G(v_{\alpha}(z))d\mu(z)\rightarrow \int_{\mathbb{C}}G(e^{-\pi|z|^2})dA(z), \quad \alpha\rightarrow+\infty.$$
	Finally,  $\frac{(\alpha+2-t^{\frac{\alpha+1}{\alpha+2}})}{(\alpha+1)t^{\frac{\alpha+1}{\alpha+2}}}\rightarrow \frac{1}{t}$ and $\frac{\alpha+1}{\pi}K_{\alpha}(t)\rightarrow 1-t+t\log t,$ as $\alpha\rightarrow +\infty,$ therefore, from \eqref{stabil} with $u$ corresponding to $f_{\alpha}(z)$, we conclude:
\begin{equation}
\label{stabilF}
\int_{\mathbb{C}}G(|f(z)|^2e^{-\pi|z|^2}) \:\!d\mu(z)\leqslant \int_{\mathbb{C}}G(e^{-\pi|z|^2}) \:\!d\mu(z)-C'(1-T)\int_{0}^{T}\frac{tG''(t)}{1-t+t\log t}dt.
\end{equation}
\subsection{Case $\alpha=-1$ and stability of functionals on Hardy space}
We can use the approach from the weighted Bergman space setting in the special case when $G'(0)\geqslant 0$ and $G$ is convex (i. e. convex and increasing function) and letting $\alpha\rightarrow -1.$ However, we can see even more. That is, since
$$\int_{A_t}(u(z)-t)_+d\mu(z)\leqslant \bigg(1-\frac{1-T}{M_{-1}(s)}\bigg)\int_{A_t*}(v_{-1}(z)-t)_+ d\mu(z),$$
for an arbitrary $t>0,$ by using $s\leqslant \pi(\frac{1}{t}-1),$ we find $M_{-1}(s)\leqslant M_{-1}(\mu(B_t))=C(1+\log\frac{1}{t}).$(This is the appropriate first case  in the Bergman setting, while the second does not exist!)
Therefore, after integration over $(0,T)$ and using
$G(u)=G'(0)u+\int_{0}^{+\infty}(u-t)_+G''(t)dt,$
we get
$$ \int_{\mathbb{D}}G(u(z)) d\mu(z)\leqslant \int_{\mathbb{D}}G(v_{-1}(z))d\mu(z)-C'(1-T)\int_{0}^{T}\frac{G''(t)}{\big(1+\log\frac{1}{t}\big)(t-1+\log\frac{1}{t})}dt.$$
Note that in case of convex $G$ this function need not be strictly increasing, but only non-decreasing! 
\begin{remark}
We use condition $G'(0)$ since, in the case of Hardy spaces, the functions $u$ and $v_{-1}$ do not have the same integral (equal to one, as in the weighted Bergman case), but instead we have the inequality $\mu(u(z)>t)\leqslant \mu(v_{-1}(z)>t)=\pi(\frac{1}{t}-1)$ and consequently $\int_{\mathbb{D}}u(z)d\mu(z)\leqslant \int_{\mathbb{D}}v_{-1}(z)d\mu(z).$ On the other hand, we work here with the convex functions and therefore only convex and increasing functions (like power functions) are covered by this approach.
\end{remark}
\subsection{Sharpness of the stability exponent} From the estimate\\
$\rho(t)\geqslant \pi\big(\big(\frac{T}{t}\big)^{\frac{1}{\alpha+2}}-1\big)$
which holds for $f$ such that $\rho(t)=0$ for $t\geqslant T$ (see the end of the second section), we have
\begin{align*}
    &\int_{\mathbb{D}}G(v_{\alpha}(z))d\mu(z)-\int_{\mathbb{D}}G(u(z))d\mu(z)=\int_{0}^{T}G'(t)(\rho_0(t)-\rho(t))dt+\int_{T}^{1}G'(t)\rho_0(t)dt\\
    &\leqslant \pi(1-T^{\frac{1}{\alpha+2}})\int_{0}^{T}G'(t)t^{-\frac{1}{\alpha+2}}dt+\pi\int_{T}^{1}G'(t)(t^{-\frac{1}{\alpha+2}}-1)dt\leqslant C_G(1-T).
\end{align*}
Therefore, if $\int_{\mathbb{D}}G(v_{\alpha}(z))d\mu(z)-\int_{\mathbb{D}}G(u(z))d\mu(z)\geqslant c_G(1-T)^{\gamma},$ we get $\gamma\geqslant 1.$ Finally, by the estimate \eqref{stabil} we conclude $\gamma=1.$

\section{Stability for convex functionals associated to a semi-definite positive operator on the weighted Bergman space. The application to the Cauchy wavelets}

Let us consider a semi-definite operator $A$ with trace $TrA=1.$ It can be represented, by the spectral theorem, as
$$A=\sum_{j=0}^{+\infty}p_j\langle \cdot, f_j\rangle f_j, \quad\text{where} \quad p_j\geqslant 0, \sum_{j=0}^{+\infty}p_j=1$$
and $f_j'$s are orthonormal functions from $\mathcal{A}^2_{\alpha}(\mathbb{D}).$ 
Its covariant symbol is
$$u(z):=\langle \kappa_z, A\kappa_z\rangle=\sum_{j=1}^{+\infty}p_j|\langle \kappa_z, f_j\rangle|^2=\sum_{j=0}^{+\infty}p_j|f_j(z)|^2(1-|z|^2)^{\alpha+2},$$
where $\kappa_z(w)=K_z(w)(1-|z|^2)^{\frac{\alpha+2}{2}}=(1-\overline{z}w)^{-\alpha-2}(1-|z|^2)^{\frac{\alpha+2}{2}}$ is the normalized Bergman reproducing kernel and 
$$f_j(z)=\langle f_j, K_z\rangle=\sum_{k=0}^{+\infty}a_{j,k}\frac{z^k}{\sqrt{c_k}} \quad \text{with} \quad \sum_{k=0}^{+\infty}|a_{j,k}|^2=1.$$
The relationship between weighted Bergman spaces and continuous Wavelet transforms associated to a family of Cauchy windows enables us to give wavelet interpretation of the results from this section (for more on wavelet transform, the reader can consult \cite{Abr}-\cite{AbSp},\cite{Be}, \cite{Da}, \cite{DaPa}). Namely, if $\psi_{\beta}$ is defined as a function from the Hardy space $H^{2}(\mathbb{C}^{+})$ such that 
$$\widehat{\psi_{\beta}}(t)=c_{\beta}^{-1}t^{\beta}e^{-t}\chi_{[0,+\infty)},$$
where $c_{\beta}=2^{2\beta-1}\Gamma(2\beta).$
Bergman transform of order $\alpha=2\beta-1$ is given by
$$B_{\alpha}f(z)=y^{-1-\frac{\alpha}{2}}W_{\overline{\psi_{\frac{\alpha+1}{2}}}}f(-x,y)=c_{\alpha}\int_{0}^{+\infty}t^{\frac{\alpha+1}{2}}\widehat{f}(t)e^{\imath z t}dx.$$
It is a unitary mapping from $H^2(\mathbb{C}^{+})$ and $\mathcal{A}^2_{\alpha}(\mathbb{C}^{+}),$ defined respectively as the subset of $L^2(\mathbb{R})$ of functions with the Fourier transform vanishing on the negative real semi-axes and the norm inhereted from $L^2$ and the space of functions with finite norm
$$\int_{\mathbb{C}^{+}}|f(z)|^2y^{\alpha}dxdy,$$
here $z:=x+\imath y.$ By its definition, we have
$$\int_{\mathbb{C^+}}|W_{\psi_{\beta}}f(x,y)|^2\frac{dxdy}{y^2}=\int_{\mathbb{C^+}}|B_{\alpha}f(z)|^2y^{\alpha}dxdy,$$
from which, after using a mapping $T_{\alpha}f(w)=\sqrt{\frac{\pi2^{\alpha}}{\alpha+1}}(1-w)^{-\alpha-2}f(\frac{w+1}{\imath(w-1)}),$ we get the relation
$$\int_{\mathbb{C^+}}|B_{\alpha}f(z)|^2y^{\alpha}dxdy=\frac{\alpha+1}{\pi}\int_{\mathbb{D}}|T_{\alpha}(B_{\alpha}f)(z)|^2(1-|z|^2)^{\alpha}dxdy,$$
and conclude that $\Pi_{\alpha}=T_{\alpha}\circ B_{\alpha}$ is a unitary map from $H^2(\mathbb{C}^+)$ to $\mathcal{A}^2_{\alpha}(\mathbb{D}).$
If we consider the semi-definite operator on $H^2(\mathbb{C}^+)$ defined as
$$\mathfrak{A}=\sum_{j=0}^{+\infty}p_j\langle \cdot, F_j\rangle F_j$$
for orthonormal $F_j'$s induced by $F_j=\Pi^{-1}_{\alpha}f_j,$ then the equivalence of Husimi functions for operators $A$ and $\mathfrak{A}$ follows from:
\begin{align*}
U(x,y)&:=\langle \psi_{\alpha, x,y}, \mathfrak{A}\psi_{\alpha, x,y}\rangle=\sum_{j=1}^{+\infty}p_j|\langle \psi_{\alpha, x,y}, F_j\rangle|^2\\
&=\sum_{j=1}^{+\infty}p_j|\langle \Pi_{\alpha}\psi_{\alpha, x,y},\Pi_{\alpha} F_j\rangle|^2=\sum_{j=1}^{+\infty}p_j|\langle \kappa_z, f_j\rangle|^2\\
&=\langle \kappa_z, A\kappa_z\rangle=u(z).
\end{align*}
Here $\psi_{\alpha, x,y}(t):=\frac{\psi_{\alpha}\big(\frac{t-x}{y}\big)}{\|\psi_{\alpha}\big(\frac{t-x}{y}\big)\|}$ is the normalized kernel of the wavelet transform.
\begin{lemma}
\label{osnovnestvari}
Let us denote $\rho(t)=\mu(\{z\in \mathbb{D}: u(z)>t\}), t>0.$ Then the following is true:\\
a) $$\Delta \log u(z)\geqslant -4(\alpha+2),$$\\
b) The function $\rho(t)$ is non-increasing and absolutely continuous on compact subsets of $(0,+\infty)$ and there holds the differential inequality
$$\rho'(t)\leqslant -\frac{\rho(t)+\pi}{(\alpha+2)t},$$
c) $T=\max_{z \in \mathbb{D}}u(z)\leqslant 1$ and $T=1$ if and only if $A=\langle \cdot, \kappa_{z_0}\rangle \kappa_{z_0}$ for some $z_0.$ For $T=1$ we have $\rho(t)=\rho_0(t)=\pi\big(t^{-\frac{1}{\alpha+2}}-1\big)$ while in the case $T<1$ there exists a unique $t^*$ such that $\rho(t)>\rho_0(t)$ for $0<t<t^*$ and $\rho(t)<\rho_0(t)$ for $t^*<t<T.$ 
\end{lemma}

\begin{proof}
We will just sketch a proof, since it uses the methods from \cite{Ku} and \cite{FrNiTi}. The part $a)$ follows from the fact that the sum of log-subharmonic functions is a log-subharmonic function and $\Delta \log (1-|z|^2)^{\alpha+2}=-4(\alpha+2).$ 
For part $c)$, we start from the co-area formula:
$$\rho'(t)=-\int_{u(z)=t}\frac{1}{|\nabla u(z)|}\frac{|dz|}{(1-|z|^2)^2},$$
which can be applied since $u(z)$ is a real-analytic function (as can easily be checked by differentiation and using Cauchy estimates on concentric disks) with $\lim_{|z|\rightarrow 1^{-}}u(z)=0$ and the set of critical points of the zero measure. \\
Using the Cauchy-Schwarz inequality, we get
\begin{equation}
\label{ksbnej}
l_h(\partial A_t)^2=\bigg(\int_{\partial A_t}\frac{|dz|}{1-|z|^2}\bigg)^2\leqslant \bigg(\int_{\partial A_t}|\nabla u(z)|^{-1}\frac{|dz|}{(1-|z|^2)^2}\bigg)\bigg(\int_{\partial A_t}|\nabla u(z)||dz|\bigg),
\end{equation}
where the second factor of the right-hand side can be determined via the Stokes formula
\begin{equation}
\label{stoks}
\int_{\partial A_t}|\nabla u(z)||dz|=-t\int_{\partial A_t}\big(\nabla u(z)\big)\cdot \nu|dz|=-t\int_{A_t}\Delta \log u(z) dxdy\leqslant -4t(\alpha+2)\rho(t).
\end{equation}
The isoperimetric inequality for the hyperbolic metric gives
\begin{equation}
\label{izopnej}
l_h(\partial A_t)^2\geqslant 4\pi\rho(t)+4\rho^2(t),
\end{equation}
therefore, we obtain
$$-\rho'(t)\geqslant \frac{l_h(\partial A_t)^2}{4(\alpha+2)t\rho(t)}\geqslant\frac{4(\pi\rho(t)+\rho^2(t))}{4(\alpha+2)t\rho(t)}=\frac{\rho(t)+\pi}{(\alpha+2)t}.$$
The first conclusion of the part $c)$ follows from the pointwise estimate $|f(z)|^2(1-|z|^2)^{\alpha+2}\leqslant \|f\|^2_{\mathcal{A}^2_{\alpha}}$ and the assumption that $\sum_{j=0}^{+\infty}p_j=1.$\\
Finally, to prove the part $c)$, let us consider the operator $T_{f,g}:=\langle \cdot, f\rangle f-\langle \cdot, g\rangle g.$ It is self-adjoint and its trace norm is equal to the sum of the absolute values of its eigenvalues. From
$$T_{f,g}(af+bg)=(a+b\langle g,f\rangle)f-(\langle f,g\rangle +b)g=\lambda (af+bg)$$
we find that $\lambda'$s are the eigenvalues of the matrix
$$\begin{bmatrix}
1 & \langle g,f\rangle\\
-\langle f,g\rangle &  -1
\end{bmatrix}
$$
i.e. we have $\lambda=\pm \sqrt{1-|\langle f,g\rangle|^2},$ thus giving $\|T_{f,g}\|_{\mathcal{S}_1}=2\sqrt{1-|\langle f,g\rangle|^2}.$ From this observation, we get:
\begin{align*}
&\|A-\langle \cdot,\kappa_z\rangle \kappa_z\|_{\mathcal{S}_1}=\|\sum_{j=1}^{+\infty}p_j\big(\langle \cdot f_j\rangle f_j- \langle \cdot, \kappa_z\rangle \kappa_z\big)\|_{\mathcal{S}_1}\\
&\leqslant \sum_{j=1}^{+\infty}p_j\|\langle \cdot f_j\rangle f_j- \langle \cdot, \kappa_z\rangle \kappa_z\|_{\mathcal{S}_1}=\sum_{j=1}^{+\infty}p_j\sqrt{1-|\langle f_j, \kappa_z\rangle|^2}\\
&\leqslant 2\sqrt{1-\sum_{j=1}^{+\infty}p_j|\langle f_j, \kappa_z\rangle|^2}=2\sqrt{1-|u(z)|^2}.
\end{align*}
Therefore, $D(A)=\inf_{z \in\mathbb{D}}\|A-\langle \cdot,\kappa_z\rangle \kappa_z\|_{\mathcal{S}_1}\leqslant 2\inf_{z \in\mathbb{D}}\sqrt{1-u(z)}=2\sqrt{1-T}.$ Note that in the special case when $A=\langle \cdot, f\rangle f$ all these inequalities become equalities and $D(A)=2\sqrt{1-T}.$\\
We see that $T=1$ implies $\inf_{z \in\mathbb{D}}\|A-\langle\cdot, \kappa_z\rangle \kappa_z\|_{\mathcal{S}_1}=0,$ i. e. $A=\langle\cdot, \kappa_{z_0}\rangle \kappa_{z_0}$ for some $z_0\in \mathbb{D},$ since rank one operators form closed subspaces of the set of bounded operators on $\mathcal{A}^2_{\alpha}(\mathbb{D})$ and the considered infimum is attained. \\

If $T=1,$ we have $A=\langle \cdot, \kappa_{z_0}\rangle \kappa_{z_0}$ and $\rho(t)=\rho_0(t)$ for $0<t<1.$ Now we turn to case $T<1.$ It can be easily checked that $\rho'(t)\leqslant -\frac{\rho(t)+\pi}{(\alpha+2)t}$ is equivalent with the condition that $\varphi(t):=t^{\frac{1}{\alpha+2}}\big(\rho(t)+\pi\big)$ is non-increasing. We note that for $\rho_0(t)=\pi(t^{-\frac{1}{\alpha+2}}-1)$ we have $\varphi_0(t):=t^{\frac{1}{\alpha+2}}\big(\rho_0(t)+\pi\big)=0$ identically and, since $\int_{0}^{1}\big(\rho(t)-\rho_0(t)\big)dt=0$ and $\varphi(t)-\varphi_0(t)$ has the same sign as $\rho(t)-\rho_0(t),$ we infer the existence of $t^*\in(0,T)$ such that $\rho(t)\geqslant\rho_0(t)$ for $0<t<t^*$ and $\rho(t)\leqslant\rho_0(t)$ for $t^*<t<T.$ Let us prove that such a $t^*$ is unique. If $\rho(t)=\rho_0(t)$ for $t \in (c,d)\subset (0,T).$ Then $\rho'(t)= -\frac{\rho(t)+\pi}{(\alpha+2)t}$ implies that all inequalities in the chain \eqref{ksbnej}, \eqref{stoks}, \eqref{izopnej} are equalities and for some $t_0$ set $A_{t_0}=\{z\in\mathbb{D}: u(z)>t_0\}$ is a disk (denote its radius by $r$); without loss of generality we can assume that its center is $0.$ Therefore, by \eqref{stoks} it follows that $\Delta \log u(z)=4(\alpha+2),$ i.e. $\Delta \log\frac{u(z)}{(1-|z|^2)^{\alpha+2}}=0.$ On $\partial A_{t_0}$ we have $u(z)=t_0,$ therefore, by the uniqueness of the solution of Dirichlet problem we have $u(z)=\frac{t_0}{(1-r^2)^{\alpha+2}}(1-|z|^2)^{\alpha+2}=C(1-|z|^2)^{\alpha+2}.$ Since $\max u(z)=T,$ this constant has the norm $1$ in $\mathcal{A}^2_{\alpha}(\mathbb{D})$ we finally conclude $C=T=1,$ which contradicts with the assumption $T<1.$
\end{proof}

Let us suppose that $u(z)\leqslant u(0)=T<1.$ From $u(0)=\sum_{j=0}^{+\infty}p_j|f_j(0)|^2=\sum_{j=0}^{+\infty}p_ja^2_{j,0}=T,$ we have
$$\sum_{j=0}^{+\infty}p_j(1-|a_{j,0}|^2)=1-T.$$
We can also assume that $a_{j,0}\in\mathbb{R}$ for all $j\geqslant 0.$
From the Cauchy-Schwarz inequality, we get:
\begin{align*}
|f_j(z)|^2&\leqslant |a_{j,0}|^2+\bigg(\sum_{k=1}^{+\infty}|a_{j,k}|^2\bigg)\bigg(\sum_{k=1}^{+\infty}\frac{|z|^{2k}}{c_k}\bigg)+2a_{j,0}\Re\bigg(\sum_{k=1}^{+\infty}a_{j,k}\frac{z^k}{\sqrt{c_k}}\bigg)\\
&\leqslant |a_{j,0}|^2+(1-|a_{j,0}|^2)((1-|z|^2)^{-\alpha-2}-1)+2a_{j,0}\Re \sum_{k=1}^{+\infty}a_{j,k}\frac{z^k}{\sqrt{c_k}},
\end{align*}
and:
\begin{align*}
u(z)(1-|z|^2)^{-\alpha-2}&\leqslant \sum_{j=0}^{+\infty}p_j|a_{j,0}|^2+\sum_{j=0}^{+\infty}p_j(1-|a_{j,0}|^2)((1-|z|^2)^{-\alpha-2}-1)\\
&+2\sum_{j=1}^{+\infty}p_j a_{j,0}\Re\sum_{k=1}^{+\infty}a_{j,k}\frac{z^k}{\sqrt{c_k}}.
\end{align*}
Here, $\frac{z^k}{\sqrt{c_k}}$ are normalized monomials whose $\mathcal{A}^2_{\alpha}$ norm is equal to one.
Denoting $\delta=\sqrt{\frac{1-T}{T}},$ $R(z)=\frac{1}{T}\sum_{j=1}^{+\infty}p_j a_{j,0}\sum_{k=1}^{+\infty}a_{j,k}\frac{z^k}{\sqrt{c_k}}$ and $h(z)=2\Re R(z),$ we have:
\begin{equation}
\label{ocenazau}
\frac{(1-|z|^2)^{-\alpha-2}}{T}\leqslant 1+\delta^2((1-|z|^2)^{-\alpha-2}-1)+h(z).
\end{equation}
From the fact that $u(z)=\sum_{j=0}^{+\infty}p_j|f_j(z)|^2(1-|z|^2)^{\alpha+2}$ has a maximum at $z=0$ we infer that the same holds for $\sum_{j=0}^{+\infty}p_j|f_j(z)|^2,$ thus giving:
$$\sum_{j=1}^{+\infty}p_j a_{j,0}\Re a_{j,1}=0.$$
Let us estimate $|h(z)|,$ $|\frac{\partial h}{\partial r}(z)|$ and $|\frac{\partial^2 h}{\partial r^2}(z)|.$ By the Cauchy-Schwarz inequality, we get:
\begin{align}
|h(z)|^2&\notag\leqslant 4|R(z)|^2\leqslant \frac{4}{T^2}\bigg(\sum_{j=1}^{+\infty}p_j|a_{j,0}|^2\bigg)\bigg(\sum_{j=1}^{+\infty}p_j\bigg|\sum_{k=2}^{+\infty}a_{j,k}\frac{z^k}{\sqrt{c_k}}\bigg|^2\bigg)\\
&\notag\leqslant \frac{4}{T^2}\cdot T\cdot \sum_{j=1}^{+\infty}p_j\bigg(\sum_{k=2}^{+\infty}|a_{j,k}|^2\bigg)\bigg(\sum_{k=2}^{+\infty}\frac{|z|^{2k}}{c_k}\bigg)\\
&\notag\leqslant\frac{4}{T}\sum_{j=1}^{+\infty}p_j(1-|a_{j,0}|^2)\big((1-|z|^2)^{-\alpha-2}-1-(\alpha+2)|z|^2\big)\\
&\label{ocenah}\leqslant 4\delta^2\big((1-|z|^2)^{-\alpha-2}-1-(\alpha+2)|z|^2\big),
\end{align}
and similarly: 
\begin{align*}
|R'(z)|^2&\leqslant \frac{1}{T^2}\bigg(\sum_{j=1}^{+\infty}p_j|a_{j,0}|^2 \bigg(\sum_{k=2}^{+\infty}a_{j,k}\frac{k|z|^{k-1}}{\sqrt{c_k}}\bigg)^2\bigg)\\
&\leqslant \frac{1}{T^2}\sum_{j=1}^{+\infty}p_j|a_{j,0}|^2\bigg(\sum_{k=2}^{+\infty}p_j|a_{j,k}|^2\bigg)\bigg(\sum_{k=2}^{+\infty}\frac{k^2|z|^{2k-2}}{c_k}\bigg)\\
&\leqslant \delta^2\sum_{k=2}^{+\infty}\frac{k^2|z|^{2k-2}}{c_k}=\delta^2(\alpha+2)\bigg(-1+(1+(\alpha+2)|z|^2)(1-|z|^2)^{-\alpha-4}\bigg)
\end{align*}
and
\begin{align*}
&|R''(z)|^2\leqslant \delta^2\sum_{k=2}^{+\infty}\frac{k^2(k-1)^2|z|^{2k-4}}{c_k}\\
&=\delta^2(\alpha+2)(\alpha+3)(1-|z|^2)^{-\alpha-6}\big(2+(\alpha+3)|z|^2(4+(\alpha+4)|z|^2)\big),
\end{align*}
thus, by $|\nabla h(z)|=2|R'(z)|$ and $|D^2h(z)|=2\sqrt{2}|R''(z)|,$ we infer:
\begin{equation}
\label{ocenaprvih}
\bigg|\frac{\partial h}{\partial r}(z)\bigg|\leqslant 2\delta\sqrt{(\alpha+2)\bigg(-1+(1+(\alpha+2)|z|^2)(1-|z|^2)^{-\alpha-4}\bigg)}
\end{equation} 
and
\begin{equation}
\label{ocenadrugih}
\bigg|\frac{\partial^2 h}{\partial r^2}(z)\bigg|\leqslant 2\sqrt{2}\delta\sqrt{(\alpha+2)(\alpha+3)(1-|z|^2)^{-\alpha-6}\big(2+(\alpha+3)|z|^2(4+(\alpha+4)|z|^2)\big)}.
\end{equation} 

This gives that $u(z)>t$ implies 
\begin{equation}
\label{gtheta}
g_{\theta}(r,\sigma):=\frac{t}{T}(1-|z|^2)^{-\alpha-2}-\delta^2((1-|z|^2)^{-\alpha-2}-1)-\sigma h(z)< 1.
\end{equation}
Therefore, we have the inclusion $\{z\in\mathbb{D}: u(z)>t\}\subset \{g_{\theta}(r,1)<1\}\\
\subset \{re^{\imath\theta}: 0\leqslant r\leqslant r_{1}(\theta)\}=:E_1,$
where $r_{\sigma}(\theta)$ is a function defined implicitly via the equation
$$g_{\theta}(r_{\sigma}(\theta), \sigma)=1.$$
It is well defined as can be seen by employing the methods from \cite{GGRT}.  Note that the sets $E_{\sigma}$ are not defined in the same way as in the paper \cite{GKMR}, since the inequality \eqref{gtheta} is slightly cruder than the analogous available in the case of sub-level sets of $|f(z)|^2(1-|z|^2)^{\alpha+2}.$ Therefore, we will sketch the proof of these properties  for $E_{\sigma}$ and prove the improved version of Lemma 3.1 from \cite{GKMR} with $u(z)$ arising from general operator $A$ and without the lower bound $t_0 \geqslant \tilde{t}_0$ (in the terminology of \cite{GKMR}, it was $c_0\geqslant \tilde{c}_0\approx 0,6194$). This will be a consequence of some changes in the second part of the proof in the spirit of \cite{GGRT}. Since, it mainly follows the method from \cite{GKMR}, we will just sketch the proof and briefly explain those parts that contain some modifications or more elegant calculations.\\
First, realize that the sets $E_{\sigma}:=\{re^{\imath\theta}: 0\leqslant r\leqslant r_{\sigma}(\theta)\}$ are compactly embedded in the unit disk and star-shaped. Namely, from
\begin{align*}
g_{\theta}(r,\sigma)&\geqslant t_0(1-r^2)^{-\alpha-2}-\delta^2\big((1-r^2)^{-\alpha-2}-1\big)-2\delta\sqrt{(1-r^2)^{-\alpha-2}-1-(\alpha+2)r^2}\\
&\geqslant (t_0-\delta^2-2\delta)(1-r^2)^{-\alpha-2}>\frac{2t_0}{9}\big(1-r^2\big)^{-\alpha-2}>0,
\end{align*}
for $\delta<\frac{t_0}{3},$ we have $r^2_{\sigma}\leqslant 1-\big(\frac{9}{2t_0}\big)^{-\frac{1}{\alpha+2}}$ and the boundedness of $E_{\sigma}$ follows.  Under the same condition ($\delta<\frac{t_0}{3}$), we infer
\begin{align*}
\frac{d}{dr}g_{\theta}(r,\sigma)&=2\big(\frac{t}{T}-\delta^2\big)(\alpha+2)r(1-r^2)^{-\alpha-3}-\frac{\partial h}{\partial r}(re^{\imath\theta})\\
&\geqslant 2(t_0-\delta^2)(\alpha+2)r(1-r^2)^{-\alpha-3}-2\delta\sqrt{(\alpha+2)\big(-1+(1+(\alpha+2)r^2)(1-r^2)^{-\alpha-4}\big)}\\
&\geqslant  2(t_0-\delta^2)(\alpha+2)r(1-r^2)^{-\alpha-3}-2\delta\sqrt{2(\alpha+2)(\alpha+3)}r(1-r^2)^{-\alpha-3}\\
&\geqslant 2(\alpha+2)r(1-r^2)^{-\alpha-3}(t_0-\delta^2-2\delta)>0,
\end{align*}
which implies that the sets are star-shaped. \\
We now turn to the estimate of the hyperbolic measure of sub-level sets for $u$ by bounding the measure of $E_1.$  
\begin{lemma}
\label{slabaocena}
For every $t_0>0$ there exists $T_0\in (t_0, 1)$ and the absolute constant $C$ such that
$$\rho(t)\leqslant \pi\big(1+Ct_0^{-3}\delta^2\big)\bigg(\bigg(\frac{T}{t}\bigg)^{\frac{1}{\alpha+2}}-1\bigg)=:\tilde{\rho}(t) , \quad  t\in (t_0, T), $$
if $f \in \mathcal{A}^2_{\alpha}$ and $\|f\|_{\mathcal{A}^2_{\alpha}}=1.$ Here $\delta=\
\sqrt\frac{1-T}{T}.$
\end{lemma}

\begin{proof}
From the formula
$$\rho(t)=\mu(\{z\in\mathbb{D}: u(z)>t\})\leqslant \mu(E_1)=\frac{1}{2}\int_{0}^{2\pi}\int_{0}^{r_{1}(\theta)}\frac{rdrd\theta}{(1-r^2)^2}=\int_{0}^{2\pi}\frac{r_1(\theta)^2d\theta}{1-r_{1}(\theta)^2},$$
after defining $F(\sigma):=\int_{0}^{2\pi}\frac{r_{\sigma}(\theta)^2d\theta}{1-r_{\sigma}(\theta)^2}$ and concluding $F'(0)=0$ from the harmonicity of $h$, by Taylor's formula, we have:
$$F(1)=F(0)+\frac{F''(s)}{2}$$
for some $s\in (0,1).$ 
Our main task becomes estimating $r'_{\sigma},$  $r''_{\sigma}$ and $F''(\sigma).$ The formula for the first derivative reads as:
$$r'_{\sigma}=\frac{h(r_{\sigma}e^{\imath\theta})}{2r_{\sigma}\frac{t}{T}(\alpha+2)(1-r^2_{\sigma})^{-\alpha-3}-2r_{\sigma}\delta^2(\alpha+2)(1-r^2_{\sigma})^{-\alpha-3}-\sigma e^{\imath \theta}\frac{\partial h}{\partial r}(r_{\sigma}e^{\imath \theta})}.$$

Let us estimate the denominator. We have:
\begin{align}
&\notag\bigg|2r_{\sigma}\frac{t}{T}(\alpha+2)(1-r^2_{\sigma})^{-\alpha-3}-2r_{\sigma}\delta^2(\alpha+2)(1-r^2_{\sigma})^{-\alpha-3}-\sigma e^{\imath \theta}\frac{\partial h}{\partial r}(r_{\sigma}e^{\imath \theta})\bigg|\\
&\notag\geqslant 2r_{\sigma}(\frac{t}{T}-\delta^2)(\alpha+2)(1-r^2_{\sigma})^{-\alpha-3}-\sigma\bigg|\frac{\partial h}{\partial r}(r_{\sigma}e^{\imath \theta})\bigg|\\
&\notag\geqslant 2r_{\sigma}(\frac{t}{T}-\delta^2)(\alpha+2)(1-r^2_{\sigma})^{-\alpha-3}-2\delta \sqrt{(\alpha+2)\big(-1+(1+(\alpha+2)r^2_{\sigma})(1-r^2_{\sigma})^{-\alpha-4}\big)}\\
&\notag \geqslant 2r_{\sigma}(1-r^2_{\sigma})^{-\alpha-3}\big((\frac{t}{T}-\delta^2)(\alpha+2)-\delta\sqrt{2(\alpha+2)(\alpha+3)}\big)\\
&\label{imenilac}\geqslant 2(\alpha+2)r_{\sigma}(1-r^2_{\sigma})^{-\alpha-3}\big(\frac{t}{T}-\delta^2-2\delta\big).
\end{align}
Using the inequality \eqref{ocenah} with $z=r_{\sigma}e^{\imath\theta},$ we get:
\begin{equation}
\label{rsigmaprvi}
|r'_{\sigma}|\leqslant \frac{\delta r_{\sigma}(1-r^2_{\sigma})}{\frac{t}{T}-\delta^2-\delta\sqrt{\frac{2(\alpha+3)
}{\alpha+2}}}\leqslant \frac{\delta r_{\sigma}(1-r^2_{\sigma})}{\frac{t}{T}-\delta^2-2\delta}\leqslant \frac{4\delta}{t_0}r_{\sigma}(1-r^2_{\sigma}).
\end{equation}
The second derivative $r''_{\sigma}$ we will write as the following sum:
\begin{align*}
r''_{\sigma}&=\frac{e^{\imath\theta}\frac{\partial h}{\partial r}(r_{\sigma}e^{\imath\theta})r'_{\sigma}}{2r_{\sigma}\frac{t}{T}(\alpha+2)(1-r^2_{\sigma})^{-\alpha-3}-2r_{\sigma}\delta^2(\alpha+2)(1-r^2_{\sigma})^{-\alpha-3}-\sigma e^{\imath \theta}\frac{\partial h}{\partial r}(r_{\sigma}e^{\imath \theta})}\\
&-\frac{2(\alpha+2)(\frac{t}{T}-\delta^2)r'_{\sigma}\big((1-r^2_{\sigma})^{-\alpha-3}+2(\alpha+3)r^2_{\sigma}(1-r^2_{\sigma})^{-\alpha-4}\big)h(r_{\sigma}e^{\imath\theta})}{(2r_{\sigma}\frac{t}{T}(\alpha+2)(1-r^2_{\sigma})^{-\alpha-3}-2r_{\sigma}\delta^2(\alpha+2)(1-r^2_{\sigma})^{-\alpha-3}-\sigma e^{\imath \theta}\frac{\partial h}{\partial r}(r_{\sigma}e^{\imath \theta}))^2}\\
&+\frac{e^{\imath\theta}\big(\sigma\frac{\partial^2 h}{\partial r^2}(r_{\sigma} e^{\imath\theta})r'_{\sigma}e^{\imath\theta}+\frac{\partial h}{\partial r}(r_{\sigma}e^{\imath\theta})\big)h(r_{\sigma}e^{\imath\theta})}{(2r_{\sigma}\frac{t}{T}(\alpha+2)(1-r^2_{\sigma})^{-\alpha-3}-2r_{\sigma}\delta^2(\alpha+2)(1-r^2_{\sigma})^{-\alpha-3}-\sigma e^{\imath \theta}\frac{\partial h}{\partial r}(r_{\sigma}e^{\imath \theta}))^2}\\
&=:T_1+T_2+T_3
\end{align*} 
and estimate it term by term. 
First, using \eqref{ocenaprvih} and 
$$(\alpha+2)\big(-1+(1+(\alpha+2)r^2_{\sigma})(1-r^2_{\sigma})^{-\alpha-4}\big)\leqslant 2(\alpha+2)(\alpha+3)r^2_{\sigma}(1-r^2_{\sigma})^{-\alpha-4},$$
along with \eqref{imenilac}, we get:
$$|T_1|\leqslant \frac{\delta\sqrt{\frac{2(\alpha+3)}{\alpha+2}}}{\frac{t}{T}-\delta^2-\delta\sqrt{\frac{2(\alpha+3)}{\alpha+2}}}|r'_{\sigma}|\leqslant \frac{2\delta}{\frac{t}{T}-\delta^2-2\delta}|r'_{\sigma}|\leqslant \frac{20\delta^2}{t_0^2}r_{\sigma}(1-r^2_{\sigma}),$$
where in the last inequality, we incorporate \eqref{rsigmaprvi}.

Similarly, \eqref{ocenah}, \eqref{imenilac} and
\begin{align*}
&\big(2(\alpha+3)r^2_{\sigma}(1-r_{\sigma})^{-\alpha-4}+(1-r^2_{\sigma})^{-\alpha-3}\big)\sqrt{(1-r^2_{\sigma})^{-\alpha-2}-1-(\alpha+2)r^2_{\sigma}}\\
&\leqslant 3(\alpha+2)r^2_{\sigma}(1-r^2_{\sigma})^{-2\alpha-6}
\end{align*}
gives
$$|T_2|\leqslant \frac{3\delta(\frac{t}{T}-\delta^2)}{(\frac{t}{T}-\delta^2-2\delta)^2}|r'_{\sigma}|\leqslant \frac{72\delta^2}{t_0^3}r_{\sigma}(1-r^2_{\sigma})$$
for $\delta<\frac{t_0}{4}.$
By \eqref{ocenah} and \eqref{ocenadrugih}
and elementary inequalities
$$(1-r^2_{\sigma})^{-\alpha-2}-1-(\alpha+2)r^2_{\sigma}\leqslant \frac{(\alpha+2)(\alpha+3)}{2}r^4_{\sigma}(1-r^2_{\sigma})^{-\alpha-2}$$
and
$$(2+4(\alpha+3)r^2_{\sigma}+(\alpha+4)(\alpha+3)r^4_{\sigma})\leqslant 3(1-r^2_{\sigma})^{-2\alpha-4,}$$
we have:
\begin{align*}
&|\frac{\partial^2 h}{\partial r^2}(r_{\sigma}e^{\imath\theta})
\big||h(r_{\sigma})||r'_{\sigma}|\leqslant 4\sqrt{2}\delta^2\big((1-r^2_{\sigma})^{-\alpha-2}-1-(\alpha+2)r^2_{\sigma}\big)^{\frac{1}{2}}\times\\
&\times\big((\alpha+2)(\alpha+3)(1-r^2_{\sigma})^{-\alpha-6}(2+(\alpha+3)r^2_{\sigma})(4+(\alpha+4)r^2_{\sigma}))\big)^{\frac{1}{2}}|r'_{\sigma}|\\
&\leqslant 4\sqrt{3}(\alpha+2)(\alpha+3)\delta^2r^2_{\sigma}(1-r^2_{\sigma})^{-2\alpha-6}|r'_{\sigma}|,
\end{align*}
while \eqref{ocenah} and \eqref{ocenaprvih} with
\begin{align*}
&\big((1-r^2_{\sigma})^{-\alpha-2}-1-(\alpha+2)r^2_{\sigma}\big)\big((1+(\alpha+2)r^2_{\sigma})(1-r^2_{\sigma})^{-\alpha-4}-1\big)\\
&\leqslant (\alpha+2)(\alpha+3)^2r^6_{\sigma}(1-r^2_{\sigma})^{-2\alpha-6}
\end{align*}
gives
$$|\frac{\partial h}{\partial r}(r_{\sigma}e^{\imath\theta})
\big||h(r_{\sigma})|\leqslant 4\delta^2(\alpha+2)(\alpha+3)r^3_{\sigma}(1-r^2_{\sigma})^{-\alpha-3} $$
and, finally, after using \eqref{rsigmaprvi}:

\begin{align*}
|T_3|&\leqslant \frac{2\sqrt{3}\delta^2|r'_{\sigma}|}{(\frac{t}{T}-\delta^2-2\delta)^2}+\frac{2\delta^2r_{\sigma}(1-r^2_{\sigma})}{(\frac{t}{T}-\delta^2-2\delta)^2}\\
&\leqslant \frac{21\sqrt{3}r_{\sigma}(1-r^2_{\sigma})\delta^2}{t_0^3}+\frac{11\delta^2r_{\sigma}(1-r^2_{\sigma})}{t^2_0}\\
&\leqslant \frac{48\delta^2}{t^3_0}r_{\sigma}(1-r^2_{\sigma}).
\end{align*}
Adding the estimates for $T_1, T_2$ and $T_3,$ we get
$$|r''_{\sigma}|\leqslant 140\frac{\delta^2}{t^3_0}r_{\sigma}(1-r^2_{\sigma}).$$
Now, from
$$F''(\sigma)=\int_{0}^{2\pi}\bigg(\frac{1+3r^2_{\sigma}}{(1-r^2_{\sigma})^3}(r'_{\sigma})^2+\frac{r_{\sigma}r''_{\sigma}}{(1-r^2_{\sigma})^2}\bigg)d\theta$$
inserting the estimates for $r'_{\sigma}$ and $r''_{\sigma},$ we get
\begin{align*}
|F''(\sigma)|&\leqslant \int_{0}^{2\pi}\bigg(\frac{64\delta^2r^2_{\sigma}}{t^2_0(1-r^2_{\sigma})^3}(1-r^2_{\sigma})^2+ \frac{140r^2_{\sigma}\delta^2}{t^3_0(1-r^2_{\sigma})^2}(1-r^2_{\sigma})  \bigg) d\theta\\
&\leqslant 204\frac{\delta^2}{t^3_0}F(\sigma),
\end{align*}
thus concluding $\rho(t)\leqslant F(0)+102\frac{\delta^2}{t^3_0}F(\sigma).$ We can compare $r_{\sigma}$ and $r_0$ through
$$\bigg|\log \frac{r_{\sigma}}{\sqrt{1-r^2_{\sigma}}}-\log \frac{r_0}{\sqrt{1-r^2_0}}\bigg|\leqslant \int_{0}^{\sigma}\frac{|r'_s|}{r_s(1-r^2_s)}ds\leqslant \sigma\frac{4\delta}{t_0}\leqslant 1,$$
therefore $\frac{r^2_{\sigma}}{1-r^2_{\sigma}}\leqslant e^2\frac{r^2_0}{1-r^2_0}$ and
$F(\sigma)\leqslant e^2F(0).$

Let us prove $F(0)\leqslant \pi\big(1+\frac{2\delta^2}{t_0}\big)\big(\big(\frac{T}{t}\big)^{\frac{1}{\alpha+2}}-1\big).$
From
$$1-r^2_0=\bigg(\frac{1-\delta^2}{\frac{t}{T}-\delta^2}\bigg)^{-\frac{1}{\alpha+2}}$$
we easily calculate
$$\frac{\mu(E_0)}{\pi}=\frac{ r^2_0}{1-r^2_0}=\bigg(\frac{1-\delta^2}{\frac{t}{T}-\delta^2}\bigg)^{\frac{1}{\alpha+2}}-1.$$
We will prove that 
$$\bigg(\frac{1-\delta^2}{\frac{t}{T}-\delta^2}\bigg)^{\frac{1}{\alpha+2}}-1\leqslant \bigg(1+\frac{2\delta^2}{t_0}\bigg)\bigg(\bigg(\frac{T}{t}\bigg)^{\frac{1}{\alpha+2}}-1\bigg).$$
After introducing the change of variables $x=\frac{T}{t}\in [1,\frac{1}{t_0}]$ it is equivalent with
$$f(x):=x^{\frac{1}{\alpha+2}}\bigg(1+\frac{2\delta^2}{t_0}\bigg)-\frac{2\delta^2}{t_0}-\bigg(\frac{1-\delta^2}{\frac{1}{x}-\delta^2}\bigg)^{\frac{1}{\alpha+2}}\geqslant 0.$$
We see that 
$$f'(x)=\frac{x^{\frac{1}{\alpha+2}-1}}{\alpha+2}\bigg(1+\frac{2\delta^2}{t_0}-\frac{(1-\delta^2)^{\frac{1}{\alpha+2}}}{(1-x\delta^2)^{\frac{1}{\alpha+2}+1}}\bigg)$$
and, in order to prove $f'(x)\geqslant 0,$ since $\frac{1-\delta^2}{1-\delta^2x}>1,$ it is enough to prove 
$$1+\frac{2\delta^2}{t_0}\geqslant \frac{1-\delta^2}{(1-\delta^2x)^2}.$$ 
We easily find 
$$\big(1+\frac{2\delta^2}{t_0}\big)(1-x\delta^2)^2\geqslant \big(1+\frac{2\delta^2}{t_0}\big)(1-\frac{\delta^2}{t_0})^2=1-\frac{3\delta^4}{t_0^2}+\frac{2\delta^6}{t_0^3}$$
and this is $\geqslant 1-\delta^2,$ since $\delta\leqslant \frac{t_0}{\sqrt{3}}.$
Therefore, $f$ increases in $[1,\frac{1}{t_0}]$ under the condition $\delta\leqslant \frac{t_0}{\sqrt{3}}$ and $f(1)=0$ implies $f(x)\geqslant 0.$
\end{proof}

\begin{remark}
In the following arguments, the existence of $t_1\in (t_0, T)$ for which $\tilde{\rho}(t)\leqslant \rho_0(t)$ in the range $t\in [t_1,T)$ will be very important. By comparing the appropriate expressions, we can find that $t_1:=\big(T^{\frac{1}{\alpha+2}}-\frac{t_0^3}{C\delta^2}(1-T^{\frac{1}{\alpha+2}})\big)^{\alpha+2}.$
\end{remark}

\subsection{First proof.} Following the arguments from \cite{FrNiTi} we find points $a,b \in (0,1), a<b$ such that $G'(b)-G'(a)>0$ and $H(\tilde{t})\leqslant\frac{I}{G'(b)-G'(a)},$ where $\tilde{t}$ is such that $G(\tilde{t})=\min\{G(a), G(b)\}.$ We divide the rest of the proof into four cases:

$1^{\circ}\quad \tilde{t}\leqslant t^*:$  From $H(t)=\int_{0}^{t}\big(\rho(\tau)-\rho_0(\tau)\big)d\tau=\int_{t}^{1}\big(\rho_0(\tau)-\rho(\tau)\big)d\tau$ we have that $-H'(t)=-t^{-\frac{1}{\alpha+2}}\big(-\pi+t^{\frac{1}{\alpha+2}}(\rho(t)+\pi)\big),$ and it increases on $t,$ since both factors (without the sign $''-''$) are positive and decreasing, So, $H$ is concave on $(0,t^*)$ and since $H(0)=0,$ the quotient $\frac{H(t)}{t}$ is non-increasing, thus giving:
$$H(t^*)\leqslant \frac{H(t^*)}{t^*}\leqslant \frac{H(\tilde{t})}{\tilde{t}}\leqslant \frac{H(\tilde{t})}{a}\leqslant \frac{I}{a(G'(b)-G'(a))}.$$ 
$2^{\circ}\quad \tilde{t}>T:$ We will just use that $H$ is decreasing:
$$\frac{I}{G'(b)-G'(a)}\geqslant H(\tilde{t})\geqslant \frac{\pi(1-\tilde{t})^2}{2(\alpha+2)}\geqslant \frac{\pi(1-b)^2}{2(\alpha+2)}\geqslant \frac{\pi(1-b)^2}{2(\alpha+2)}(1-T).$$
Now fix some $t_0\in (0,1).$ \\
$3^{\circ}\quad \tilde{t} \in (t^*, t_1),$ where $t_1$ is as defined in the remark after the previous lemma: We use the estimate
$$H(t_1)\geqslant\int_{t_1}^{T}\big(\rho_0(\tau)-\rho(\tau)\big)d\tau\geqslant c(1-T)$$
and the monotonicity of $H,$ thus getting $H(\tilde{t})\geqslant H(t_1)\geqslant c(1-T).$\\
$4^{\circ}\quad  \tilde{t}\in (t_1, T):$ We have $H(t)\geqslant \int_{t}^{1}\tau^{\frac{1}{\alpha+2}}\big(\rho_0(\tau)-\rho(\tau)\big)d\tau=:h(t)$
and $h'(t)=t^{\frac{1}{\alpha+2}}\big(\rho(\tau)-\rho_0(\tau)\big),$ which is decreasing, we conclude that $h$ is concave. Define:
$$g(t)=\begin{cases}
h(t), & t \in [t_1, T],\\
h(T)+h'(T)(t-T),& t \in (T, t_2],
\end{cases}$$
where $t_2$ is such that $h(T)+h'(T)(t_2-T)=0.$ More accurate, we have $t_2=T+\frac{\int_{T}^{1}(1-\tau^{\frac{1}{\alpha+2}})d\tau}{1-T^{\frac{1}{\alpha+2}}}=\frac{1-T^{\frac{\alpha+3}{\alpha+2}}}{(\alpha+3)\big(1-T^{\frac{1}{\alpha+2}}\big)}$ and it belongs to the interval $[T,1].$
Since $\lim_{T\rightarrow 1}t_2=1$ there exists $c>0$ such that for $1-T<c$ we have $t_2-b>\frac{1-b}{2}.$ The function $g$ is concave and we get
\begin{align*}
g(\tilde{t})&=g(\tilde{t})-g(t_2)=\frac{g(\tilde{t})-g(t_2)}{t_2-\tilde{t}}\cdot(t_2-\tilde{t})\\
&\geqslant \frac{g(t_1)-g(t_2)}{t_2-t_1}\cdot(t_2-\tilde{t})=\frac{g(t_1)}{t_2-t_1}\cdot(t_2-\tilde{t})\\
&>\frac{1-b}{2}g(t_1).
\end{align*}
By Chebychev's integral inequality, we conclude
\begin{align*}
\int_{t_1}^{1}\tau^{\frac{1}{\alpha+2}}\big(\rho_0(\tau)-\rho(\tau)\big)d\tau &\geqslant(1-t_1) \frac{\int_{t_1}^{1}\big(\rho_0(\tau)-\rho(\tau)\big)d\tau}{\int_{t_1}^{1}\tau^{-\frac{1}{\alpha+2}}d\tau}\\
&=\frac{\alpha+1}{\alpha+2}\big(1-t_1^{\frac{\alpha+1}{\alpha+2}}\big)^{-1}(1-t_1) \int_{t_1}^{1}\big(\rho_0(\tau)-\rho(\tau)\big)d\tau\\
&\geqslant \frac{1-t_1}{ \log\frac{1}{t_1}}\int_{t_1}^{1}\big(\rho_0(\tau)-\rho(\tau)\big)d\tau\\
\end{align*}
Recall that $c$ is chosen such that for $T>1-c$  $t_2-b>\frac{1-b}{2}$ is satisfied, therefore $t_1$ which depends on $T$ is at a positive distance from $0!$ \\
If $T\leqslant 1-c$ then $H(\tilde{t})\geqslant H(T)\geqslant \frac{\pi}{\alpha+2}(1-T)^2\geqslant \frac{\pi c}{\alpha+2}(1-T),$ and the conclusion follows. \\
\begin{remark} (Weak stability for $\mathcal{A}^p_{\alpha}$)
For $f \in \mathcal{A}^p_{\alpha}$ and the hyperbolic measure of the super-level set of the corresponding function $u,$ one can prove that:
\begin{align*}
H(t)&=\int_{t}^{1}(\rho_0(\tau)-\rho(\tau))d\tau=\int_{t}^{1}\rho_0(\tau)d\tau\\
&=\pi\int_{t}^{1}\big(\tau^{-\frac{1}{\alpha+2}}-1\big)d\tau=\frac{\pi}{\alpha+1}\big((\alpha+1)t+1-(\alpha+2)t^{\frac{\alpha+1}{\alpha+2}}\big)\\
&\geqslant \frac{\pi}{2(\alpha+2)}(1-t)^2,
\end{align*}
thus, getting
$$H(t^*)\geqslant \frac{\pi}{2(\alpha+2)}\big(1-t\big)^2\geqslant \frac{\pi}{2(\alpha+2)}\big(1-T\big)^2,$$ 
where the above approach finally gives the estimate $I\geqslant c(1-T)^2,$ with analogously defined $I$ and $T$. More general and very elegant and instructive proof of the stability of these estimates for $\mathcal{A}^p_{\alpha}$ is given in \cite{NiRiTi}. One can notice that the method of proof of the stability of concentration estimate from \cite{GKMR} can be verbatim repeated for $\rho(t)$ and $\rho_0(t)$ as well as for their inverses, hence the main result of \cite{GKMR} holds in this general setting. 
\end{remark}
\subsection{Second proof.} We follow \cite{NiRiTi}. Fix $t_0\in (0,T)$ and $T\geqslant \frac{2}{3}.$ By using decomposition $G(t)=G_1(t)+(G(t)-G_1(t)),$ where $G_1(t)$ is a convex function
$$
G_1(t)=
\begin{cases}
G(t), 0\leqslant t\leqslant \tau_1,\\
G'_{-}(t)(t-\tau_1)+G(\tau_1), \tau_1\leqslant t\leqslant 1.
\end{cases}
$$
Hence, we have:
\begin{align*}
&\int_{\mathbb{D}}G(v_{\alpha}(z))d\mu(z)-\int_{\mathbb{D}}G(u(z))d\mu(z)\\
&=\int_{0}^{1}G_1'(t)\rho_0(t)dt-\int_{0}^{\tau_1}G_1'(t)\rho(t)dt+\int_{\tau_1}^{1}(G'(t)-G_{-}'(\tau_1))(\rho_0(t)-\rho(t))dt\\
&\geqslant \int_{\tau_1}^{1}(G'(t)-G_{-}'(\tau_1))(\rho_0(t)-\rho(t))dt.
\end{align*}
by Kulikov's inequality. This holds for every $\tau_1\leqslant t\leqslant T.$ 
From the lemma \ref{slabaocena}, we have:
$$\rho(t)\leqslant \pi\big(1+C_0(1-T))\bigg(\bigg(\frac{T}{t}\bigg)^{\frac{1}{\alpha+2}}-1\bigg)=(1+C_0(1-T))\rho_0\bigg(\frac{t}{T}\bigg)=:\tilde{\rho}(t) , \quad  t\in (t_0, T), $$
with $C_0:=\frac{3Ct_0^{-3}}{2}.$ Suppose that there are $t_0<\tau_1<\tau_2<\tau_3\leqslant T$ such that the estimates 
\begin{align}
\notag &\rho_0(t)\geqslant \tilde{\rho}(t), \quad \tau_1\leqslant t\leqslant T, \\
\label{nejednakosti} &\rho_0(t)\geqslant \tilde{\rho}(t)+c(1-T), \quad \tau_1\leqslant t\leqslant \tau_2<\tau_3\leqslant T,
\end{align}
are available for some $c>0.$ Then the conclusion follows from:
\begin{align*}
 &\int_{\tau_1}^{1}(G'(t)-G_{-}'(\tau_1))(\rho_0(t)-\rho(t))dt\geqslant \int_{\tau_1}^{\tau_2}(G'(t)-G_{-}'(\tau_1))(\rho_0(t)-\tilde{\rho}(t))dt\\
 &\geqslant c(1-T)\int_{\tau_1}^{\tau_2}(G'(t)-G_{-}'(\tau_1))dt=C'(1-T).
\end{align*}
To see that \eqref{nejednakosti} holds for some $\tau_1<\tau_2<\tau_3\leqslant T,$ we denote that:
\begin{align}
&\notag \rho_0(t)-\rho_0\bigg(\frac{t}{T}\bigg)-C_0(1-T)\rho_0\bigg(\frac{t}{T}\bigg)\\
&\label{t123}\geqslant \frac{t(1-T)}{T}\rho_0'(t)-C_0(1-T)\rho_0\bigg(\frac{t}{T}\bigg)=(1-T)\bigg(\frac{t}{T}-C_0\rho_0\bigg(\frac{t}{T}\bigg)\bigg)\geq c_1(1-T),
\end{align}
with $c_1$ equal to $0$ or $c$. This is possible since $\phi(x)=x-C_0\rho_0(x)$ is increasing and equal to $1$ for $x=1,$ therefore, the positive or strictly positive for $x$ close enough to $1.$\\
 
 \section{Stability of the estimate for the weighted Bergman spaces in the unit ball in $\mathbb{C}^n$}

In this section, we sketch a proof of the sharp quantitative estimate for convex functionals in weighted Bergman spaces in the unit ball $\mathbb{B}_n\subset \mathbb{C}^n$. Define analogously the case of one complex variable: 
$\rho(t)=\mu_n(\{z\in\mathbb{B}_n: u(z)>t\})$ and 
$\rho_0(t)=\mu_n(\{z\in\mathbb{B}_n: (1-|z|^2)^{\alpha+n+1}>t\})$, where $u(z)=|f(z)|^2(1-|z|^2)^{\alpha+n+1}$.
The distribution $\rho_0$ can be explicitly calculated:
$$\rho_0(t)=\int_{(1-|z|^2)^{\alpha+n+1}>t}\frac{dV(z)}{(1-|z|^2)^{n+1}}=\frac{\omega_n}{2n}(t^{-\frac{1}{\alpha+n+1}}-1)^n.$$
Here $\omega_n$ is the surface area of the unit sphere in $
\mathbb{S}^{2n-1}.$ If $\max_{z\in \mathbb{B}_n}u(z)=T=u(0),$ we write 
$$\frac{f(z)}{\sqrt{T}}=1+R(z),\quad \text{where } R(z)=\sum_{|\beta|\geqslant 2}\frac{a_{\beta}z^{\beta}}{\sqrt{c_{\beta}T}}, \quad \text{and}\quad  z^{\beta}:=z_1^{\beta_1}z_2^{\beta_2}\cdot\dots z_n^{\beta_n}.$$ 
We estimate $R(z)$ and its partial derivatives by the Cauchy-Schwarz inequality
$$|R(z)|^2\leqslant \big|\sum_{|\beta|\geqslant 2}\frac{a_{\beta}}{\sqrt{c_{\beta}T}}\big|^2\leqslant \frac{1}{T}\sum_{|\beta|\geqslant 2}|a_{\beta}|^2\sum_{|\beta|\geqslant 2}\frac{|z^{2\beta}|}{c_{\beta}}
=\frac{1-T}{T}\sum_{|\beta|\geqslant 2}\frac{|z^{2\beta}|}{c_{\beta}},$$
where the last sum is equal to
\begin{align*}
&\sum_{|\beta|\geqslant 2}\frac{|z_1|^{2\beta}|z_2|^{2\beta_2}\dots |z_n|^{2\beta_n}}{\Gamma(1+\beta_1)\Gamma(1+\beta_2)\dots \Gamma(1+\beta_n)}\cdot \frac{\Gamma(n+\alpha+1+\beta_1+\beta_2+\dots+\beta_n)}{\Gamma(n+\alpha+1)}\\
&=\sum_{k=2}\frac{\Gamma(n+\alpha+k+1)}{\Gamma(n+\alpha+1)}\bigg(\sum_{|\beta|=k}\frac{k!}{\beta_1!\beta_2!\dots \beta_n!}|z_1|^{2\beta}|z_2|^{2\beta_2}\dots |z_n|^{2\beta_n}\bigg)\\
&=\sum_{k=2}\frac{\Gamma(n+\alpha+k+1)}{\Gamma(n+\alpha+1)}|z|^k=(1-|z|^2)^{-n-\alpha-1}-1-(\alpha+n+1)|z|^2=:g(|z|^2).
\end{align*}
Similarly, 
$$\bigg|\frac{\partial R}{\partial z_i}\bigg|^2\leqslant \bigg|\sum_{|\beta|\geqslant 2}\frac{a_{\beta}}{\sqrt{c_{\beta}T}}\beta_i z^{\beta-e_i}\bigg|^2\leqslant \frac{1-T}{T}\sum_{|\beta|\geqslant 2} \frac{\beta_i^2|z^{2\beta-2e_i}|}{c_{\beta}}.$$
Denote $h_i=|z_i|^2.$ Then 
\begin{align*}
&\sum_{|\beta|=k\geqslant 2}\frac{\Gamma(n+\alpha+k+1)}{\Gamma(n+\alpha+1)}\bigg(\sum_{\beta_i=1}\frac{k!}{\beta_1!\beta_2!\dots \beta_n!}\beta_i^2h_1^{\beta_1}\dots h_i^{\beta_i-1}\dots h_n^{\beta_n}\bigg)\\
&=\frac{d}{dh_i}g(h_1+h_2+\dots+h_n)+h_i\frac{d^2}{dh_i^2}g(h_1+h_2+\dots+h_n)\\
&=(\alpha+n+1)\big((1-|z|^2)^{-\alpha-n-2}-1\big)+(\alpha+n+1)(\alpha+n+2)|z_i|^2(1-|z|^2)^{-\alpha-n-3}\\
&\leqslant 2(\alpha+n+1)(\alpha+n+2)|z|^2(1-|z|^2)^{-\alpha-n-3}=:a(|z|^2)
\end{align*}
Using these estimates, we get
\begin{align*}
 t&<|f(z)|^2(1-|z|^2)^{\alpha+n+1}=T(1-|z|^2)^{\alpha+n+1}(1+|R(z)|^2+2\Re R(z)) \\
 &\leqslant T(1-|z|^2)^{\alpha+n+1}\big(1+\delta^2((1-|z|^2)^{-\alpha-n-1}-1-(\alpha+n+1)|z|^2)+2\Re R(z)\big)
\end{align*}
and therefore, $\{z\in\mathbb{B}_n: u(z)>t\}\subset E_{\sigma},$ where 
$g_{\omega}(r,\sigma)=\frac{t}{T}(1-r^2)^{-\alpha-n-1}-\delta^2g(r)-2\sigma \Re R(r\omega)$ and
$E_{\sigma}=\{g_{\omega}(r,\sigma)<1\}.$ Using the estimates for $R,$ as in the one-dimensional complex case, one can see $g_{\omega}(r,\sigma)\geqslant \big(t_0^2-\delta^2-2\delta\big)(1-r^2)^{-\alpha-n-1}$
for $\delta\leqslant \frac{t_0}{3}.$ Consequently $r_{\sigma}^2\leqslant 1-\big(\frac{2t_0}{9}\big)^{-\frac{1}{\alpha+n+1}}$ and $E_{\sigma}$ are compactly embedded in $\mathbb{B}_n.$
Analogously, from
\begin{align*}
&\frac{d}{dr}g_{\omega}(r,\sigma)=\frac{2rt}{T}(\alpha+n+1)(1-r^2)^{-\alpha-n-2}-\delta^2\frac{d}{dr}g(r)-2\sigma \frac{d}{dr}\Re R(r\omega)\\
&\geqslant \frac{2rt}{T}(\alpha+n+1)(1-r^2)^{-\alpha-n-2}-2r\delta^2(\alpha+n+1)(1-r^2)^{-\alpha-n-2}-2\sqrt{2n}\delta\sqrt{a(r^2)}\\
&\geqslant 2r(\alpha+n+1)(\alpha+n+2)(1-r^2)^{-\alpha-n-2}\bigg(\frac{t}{T}-\delta^2-8\sqrt{2n}\delta\bigg)>0
\end{align*}
for $\delta\leqslant \frac{t_0}{9\sqrt{2n}},$ we conclude that these sets are star-shaped. Their hyperbolic measure can be calculated via
\begin{align*}
&\mu_n(\{\omega\in \mathbb{S}^{2n-1},r \in [0,r_{\sigma}(\omega)]\})\\
&=\omega_n\int_{\mathbb{S}^{2n-1}}\bigg(\int_{0}^{r_{\sigma}(\omega)}\frac{r^{2n-1}}{(1-r^2)^{n+1}}\bigg)dS(\omega)\\
&=\frac{\omega_n}{2n}\int_{\mathbb{S}^{2n-1}}\frac{r_{\sigma}^{2n}(\omega)}{(1-r_{\sigma}^2(\omega))^{n}}dS(\omega),
\end{align*}
where $g_{\omega}(r_{\sigma},\sigma)=1.$  Define $F(\sigma):=\frac{\omega_n}{2n}\int_{\mathbb{S}^{2n-1}}\frac{r_{\sigma}^{2n}(\omega)}{(1-r_{\sigma}^2(\omega))^{n}}dS(\omega).$ From
$$F'(\sigma)=\omega_n\int_{\mathbb{S}^{2n-1}}\frac{r_{\sigma}^{2n-1}(\omega)\frac{d}{d\sigma}r_{\sigma}(\omega)}{(1-r_{\sigma}^2(\omega))^{n+1}}dS(\omega)$$
we have
\begin{align*}
F'(0)&=\omega_n\int_{\mathbb{S}^{2n-1}}\frac{r_0^{2n-1}(\omega)}{(1-r_0^2(\omega))^{n+1}}h(r_0(\omega))dS(\omega)\\
&=\frac{\omega_nr_0^{2n-1}}{(1-r_0^2)^{n+1}}\int_{\mathbb{S}}h(r_0(\omega))dS(\omega)
=\frac{\omega_nr_0^{2n-1}h(0)}{(1-r_0^2)^{n+1}}=0,
\end{align*} 
because of (pluri-)harmonicity of $h:=\Re R(z).$ 
Also
$$F''(\sigma)=\omega_n\int_{\mathbb{S}^{2n-1}}\bigg(\frac{r_{\sigma}^{2n-1}(\omega)}{(1-r_{\sigma}^2(\omega))^{n+1}}r''_{\sigma}(\omega)+\frac{r_{\sigma}^{2n-2}(\omega)\big(2n-1+3r_{\sigma}^2(\omega)\big)}{(1-r_{\sigma}^2(\omega))^{n+2}}(r'_{\sigma}(\omega))^2\bigg)dS(\omega).$$
Therefore, $F(s)=F(0)+\frac{1}{2}F''(\xi),$ for $\xi \in (0,s).$ The rest of the proof follows by tedious calculations similar to those in the fourth section - we have to find the estimates for the second derivatives of $R,$ then for $r'_{\sigma}, r''_{\sigma},$ for $F''(\sigma)$ in terms of $F(\sigma)...$\\ 
Therefore, we have the following
\begin{lemma}
\label{slabaocena}
For every $\alpha>-1$, $f \in \mathcal{A}^2_{\alpha}(\mathbb{B}_n)$ and $t_0 \in (0,1),$ there exist $T_0\in (t_0, 1)$ and the absolute constant $C'$ depending only on $t_0$ and $n$  such that
$$\rho(t)\leqslant \frac{\omega_n}{2n}\big(1+C'\delta^2\big)\bigg(\bigg(\frac{T}{t}\bigg)^{\frac{1}{\alpha+n+1}}-1\bigg)^n=:\tilde{\rho}(t) , \quad  t\in (t_0, T), $$
if $f \in \mathcal{A}^2_{\alpha}(\mathbb{B}_n)$ and $\|f\|_{\mathcal{A}^2_{\alpha}}(\mathbb{B}_n)=1$ and $\delta=\
\sqrt\frac{1-T}{T}.$
\end{lemma}
Since $\rho_0$ is again a convex and a decreasing function, we have the analogs of \eqref{t123} and, therefore, 
\eqref{nejednakosti} hold for some $\tau_1\leqslant t\leqslant \tau_2<\tau_3\leqslant T$ for $T$ close to $1.$ Along with Theorem 5.2. from \cite{NiRiTi}, we finish a proof of Theorem .

\section{Hardy space counterpart} 
In \cite{GKMR} the sharp quantitative form of the concentration inequality is given for the Hardy space function. Here, we will give a short proof of the main local inequality using only the estimates for hyperbolic measure of super-level sets of $u$. By the bathtub principle, we have:
$$\int_{\Omega}u(z)d\mu(z)\leqslant \int_{u(z)>t}u(z)d\mu(z) $$
for $\mu(\Omega)=\mu(\{z \in \mathbb{D}: u(z)>t\}).$
From
\begin{align*}
&\int_{u(z)>t}u(z)d\mu(z)=\int_{u(z)>t}\bigg(\int_{0}^{u(z)}d\tau\bigg)d\mu(z)\\
&=\int_{0}^{T}\bigg(\int_{u(z)>\max\{t,\tau\}}d\mu(z)\bigg)d\tau\\
&=\int_{0}^{t}\bigg(\int_{u(z)>t}d\mu(z)\bigg)d\tau+\int_{t}^{T}\bigg(\int_{u(z)>\tau}d\mu(z)\bigg)d\tau=t\rho(t)+\int_{t}^{T}\rho(\tau)d\tau.
\end{align*}
Let us denote $\Phi(t):=t\rho(t)+\int_{t}^{T}\rho(\tau)d\tau-\pi\log\big(1+\frac{\rho(t)}{\pi}\big), \rho(t)=\mu(\{z\in \mathbb{D}: u(z)>t\}).$ Differentiating, we get: $\Phi'(t)=\rho'(t)\big(t-\frac{\pi}{\pi+\rho(t)}\big),$ and, since $\rho'(t)<0$ and $\rho(t)\leqslant \pi\big(\frac{1}{t}-1\big),$ we get that $\Phi$ is increasing and, by $\Phi(T)=0,$ therefore $\Phi(t)\leqslant 0$, i.e.:
$$\int_{u(z)>t}u(z)d\mu(z)\leqslant \pi\log\bigg(1+\frac{\rho(t)}{\pi}\bigg).$$

In order to prove the stability estimate, we appeal to the following refined inequality for the distribution function of $u:$
$$\rho(t)\leqslant \pi\bigg(1+\frac{C(1-T)}{t_0^3T}\bigg)\bigg(\frac{T}{t}-1\bigg)=\tilde{\rho}(t),$$
for some $C>0$ and $\tau\in [t_1, T],$ where $t_1=T(1-\frac{t_0^3}{C})$ is the specific value defined such that $\tilde{\rho}(t)\leqslant \rho_0(t)$ for $t\in (t_1, T).$ 
 
We start from the estimate 
\begin{align*}
I&=\int_{0}^{1} G'(\tau)(\rho_0(\tau)-\rho(\tau))d\tau\\
&\geqslant \pi\int_{t_1}^{T}G'(\tau)\bigg(\frac{1}{\tau}-1-\bigg(1+\frac{C(1-T)}{t_0^3T}\bigg)\bigg(\frac{T}{\tau}-1\bigg)\bigg)d\tau\\
&=\pi(1-T)\int_{t_1}^{T}\frac{G'(\tau)}{\tau}\bigg(1-\frac{C}{t_0^3T}\big(T-\tau\big)\bigg)d\tau=c_G(1-T).
\end{align*}
Here $t_1=T(1-\frac{t_0^3}{C})$ is the lower estimate for $t$ for which the expression inside the brackets is positive. This estimate holds for $T\geqslant T_0.$ 

For small values of $T,$ we use the following chain of inequalities 
\begin{align*}
&\int_{T}^{1}G'(\tau)(\rho_0(\tau)-\rho(\tau))d\tau=\int_{T}^{1}G'(\tau)\rho_0(\tau)d\tau\\
&=\pi\int_{T}^{1}G'(\tau)\bigg(\frac{1}{\tau}-1\bigg)d\tau\geqslant \pi \int_{T}^{1-\gamma(1-T)}G'(\tau)\bigg(\frac{1}{\tau}-1\bigg)d\tau\\
&\geqslant\pi\bigg(\frac{1}{1-\gamma(1-T)}-1\bigg)\int_{T}^{1-\gamma(1-T)}G'(\tau)d\tau\\
&=\pi\frac{\gamma(1-T)}{1-\gamma(1-T)}\bigg(G\big(1-\gamma(1-T)\big)-G(T)\bigg).
\end{align*}
Choosing $\gamma$ such that $1-\gamma(1-T)=\frac{1+T_0}{2},$ we get
\begin{align*}
&\int_{T}^{1}G'(\tau)(\rho_0(\tau)-\rho(\tau))d\tau\geqslant \pi\frac{1-T_0}{1+T_0}\bigg(G\bigg(\frac{1+T_0}{2}\bigg)-G(T)\bigg)\\
&\geqslant \pi\frac{1-T_0}{1+T_0}\bigg(G\bigg(\frac{1+T_0}{2}\bigg)-G(T_0)\bigg)\geqslant\pi\frac{1-T_0}{1+T_0}\bigg(G\bigg(\frac{1+T_0}{2}\bigg)-G(T_0)\bigg)(1-T).
\end{align*}\\
Therefore, in any case:
$$I=\int_{0}^{1}G'(t)(\rho_0(t)-\rho(t))dt\geqslant c_G'(1-T)$$
for a strictly increasing function $G.$ The reader can easily see that proving result for a general non-strictly increasing function strongly depends on the availability of better estimates for $\rho(t),$ therefore the following questions naturally arise: Does there exist a strictly decreasing function $h,$ such that for every $f \in H^2(\mathbb{D})$ with $\sup_{z \in \mathbb{D}}|f(z)|^2(1-|z|^2)\leqslant T\|f\|_{H^2}, T<1$ there holds $\rho(t)\leqslant h(t,T)?$ Can it be explicitly found? 
\begin{remark}
One can note that it is possible to give the multidimensional Hardy space analog of Theorem 4. Namely, the appropriate inequality for any increasing function $G$ will hold if and only if $\rho(t)\leqslant \rho_0(t)$ (as defined in Section 5), and the Lemma \ref{slabaocena} for $\alpha=-1$ will infer the result. 
\end{remark}


\section{Acknowledgements} The author is partially supported by MPNTR grant no. 174017.\\


\bibliographystyle{amsalpha}

\end{document}